\documentclass[11pt]{amsart}
\usepackage{amsmath}
\usepackage{amssymb}
\usepackage{amsfonts}
\usepackage{mathrsfs}
\usepackage{tikz}
\usepackage{color}
\usepackage{pstricks}
\usepackage{psfrag}

\newtheorem{theorem}{Theorem}[section]
\newtheorem{lemma}[theorem]{Lemma}

\newtheorem{remark}[theorem]{Remark}

\newtheorem{corollary}[theorem]{Corollary}

\newcommand{\Rbb}{\mathbb{R}}
\newcommand{\Nbb}{\mathbb{N}}
\newcommand{\Zbb}{\mathbb{Z}}

\newcommand{\ploi}{PL_+(I)}
\newcommand{\supp}[1]{\mathsf{Supp}(#1)}
\newcommand{\defeq}{:=}
\newcommand{\msc}[1]{\mathscr{#1}}
\newcommand{\breaks}[1]{\msc{B}_{#1}}

\newcommand{\Top}{Top}
\newcommand{\Lower}{Lower}
\newcommand{\maxDepth}{\operatorname{maxDepth}}

\newcommand{\siggp}{\mathsf{sigGrp}}
\newcommand{\orbd}{\mathsf{orbDepth}}

\newcommand{\seen}{\msc{S}}

\newcommand{\unseen}{\msc{U}}
\newcommand{\counter}{\mathsf{counter}}

\newcommand{\ones}{one-sided overlap}
\newcommand{\bado}{complex overlap}  
\newcommand{\nice}{one-bump functions with fundamental domains} 

\begin{document}

\title[Determining solubility of groups
   of homeomorphisms]
{Determining solubility for finitely generated groups
   of PL homeomorphisms}

\author[Collin Bleak]{Collin Bleak}
\address{School of Mathematics and Statistics\\
        University of St.~Andrews\\
         North Haugh St Andrews, Fife KY16 9SS, Scotland }
\email{cb211@st-andrews.ac.uk}

\author[Tara Brough]{Tara Brough}
\address{School of Mathematics and Statistics\\
        University of St.~Andrews\\
         North Haugh St Andrews, Fife KY16 9SS, Scotland}
\email{tarabrough@gmail.com}

\author[Susan Hermiller]{Susan Hermiller}
\address{Department of Mathematics\\
        University of Nebraska\\
         Lincoln NE 68588-0130, USA}
\email{hermiller@unl.edu}



\begin{abstract}
The set of finitely generated subgroups of
the group $\ploi$ of orientation-preserving piecewise-linear homeomorphisms 
of the unit interval includes many important 
groups, most notably R.~Thompson's group $F$.
In this paper we show that every finitely generated subgroup $G<\ploi$ is 
either soluble, or contains an embedded copy of Brin's group $B$, a finitely 
generated, non-soluble group, which verifies a conjecture of the first author 
from 2009.  In the case that $G$ is soluble, we show that the derived length 
of $G$ is bounded above by the number of breakpoints 
of any finite set of generators.  We specify a set of `computable'
subgroups of $\ploi$ (which includes R. Thompson's group $F$) and we give 
an algorithm which determines in finite time whether or not any given finite 
subset $X$ of such a computable group generates a soluble group. When the 
group is soluble, the algorithm also determines the derived length of 
$\langle X\rangle$.  Finally, we give a solution of the
membership problem for a family of finitely generated soluble
subgroups of any computable subgroup of $\ploi$.
\end{abstract}

\maketitle


\section{Introduction}\label{sec:intro}


In \cite{brinU,bpgsc,bpasc,bpnsc} a theory is built 
connecting the solubility class of a subgroup $G$ 
of the group $\ploi$ of piecewise-linear  
orientation-preserving homeomorphisms of $I=[0,1]$ 
(with finitely many breaks in slope) with 
data on how the supports of the elements of $G$ overlap 
with each other, and how these supports relate 
to the support of the whole action of $G$ on $I$.  
One result in that theory is that there is a non-soluble 
group $W$ which is not finitely generated, and which 
contains an embedded copy of every soluble subgroup of $\ploi$, 
such that any non-soluble subgroup of $\ploi$ contains an 
embedded copy of $W$ (Corollary 1.2 and Theorem 1.1, 
respectively, of \cite{bpnsc}).  However, in the finitely 
generated case, it has been believed that one could considerably 
strengthen that result.  Indeed, it is conjectured in \cite{bpnsc} 
that any finitely generated non-soluble subgroup of $\ploi$ 
contains an embedded copy of Brin's group $B$, a two-generated 
non-soluble group introduced by Brin in Section 5 of \cite{brinEA} 
as $G_1$.  In this paper we verify this conjecture.

To be more specific, we say that $G$ \emph{admits
a transition chain} if there are
two elements $g,h \in G$ with components of
support $(a,b),(c,d)$, respectively, such that
$a<c<b<d$.  If $G$ does not admit such a 
chain, then we say that $G$ is \emph{chainless}.
In a similar fashion, the group
$G$ \emph{admits a \ones} if there
are two elements $g,h \in G$ with components of 
support $(a,b), (a,c)$ or
$(a,b),(c,b)$, respectively, such that $a<c<b$.
Finally, $G$ \emph{admits a tower of infinite height}
if there is an infinite sequence
$\{g_i\}_{i \in \Nbb}$ of elements of
$G$ with components of support $A_i$, respectively,
such that $A_{i+1} \subsetneq A_i$ for all $i$.
We illustrate these three properties
in the special case of elements with
a single component of support in 
Figure~\ref{fig:threeSits}.

In~\cite[Theorem~1.1 and Lemma~1.4]{bpgsc}
(restated with further details
in Theorem~\ref{thm:towerSolveClass} 
and Lemma~\ref{lem:weak-tower}
of Section~\ref{subsec:orbitals-towers-chains} below) 
the first author shows that a soluble subgroup of
$\ploi$ must be chainless and does not
admit a tower of infinite height.  
In our main
theorem (Theorem~\ref{main})
in Section~\ref{sec:main} below,
we show that in the finitely generated
case the converse of each of these also holds.
In Lemma~\ref{orbits}
we show that for a subgroup $G \leq \ploi$
generated by a finite set $X \subset \ploi$, the
number of $G$-orbits of the set of breakpoints 
of elements of $G$ 
is bounded above by the cardinality of 
the set of breakpoints of the elements of $X$
(a point $x\in (0,1)$ is a \emph{breakpoint} of $g\in\ploi$
if $g$ changes slope at $x$);
this is applied in Theorem~\ref{main} to obtain a
bound on the derived length
in the soluble case.

\begin{center}
\begin{figure}\label{fig:threeSits}
\includegraphics[height=100pt,width=100pt]{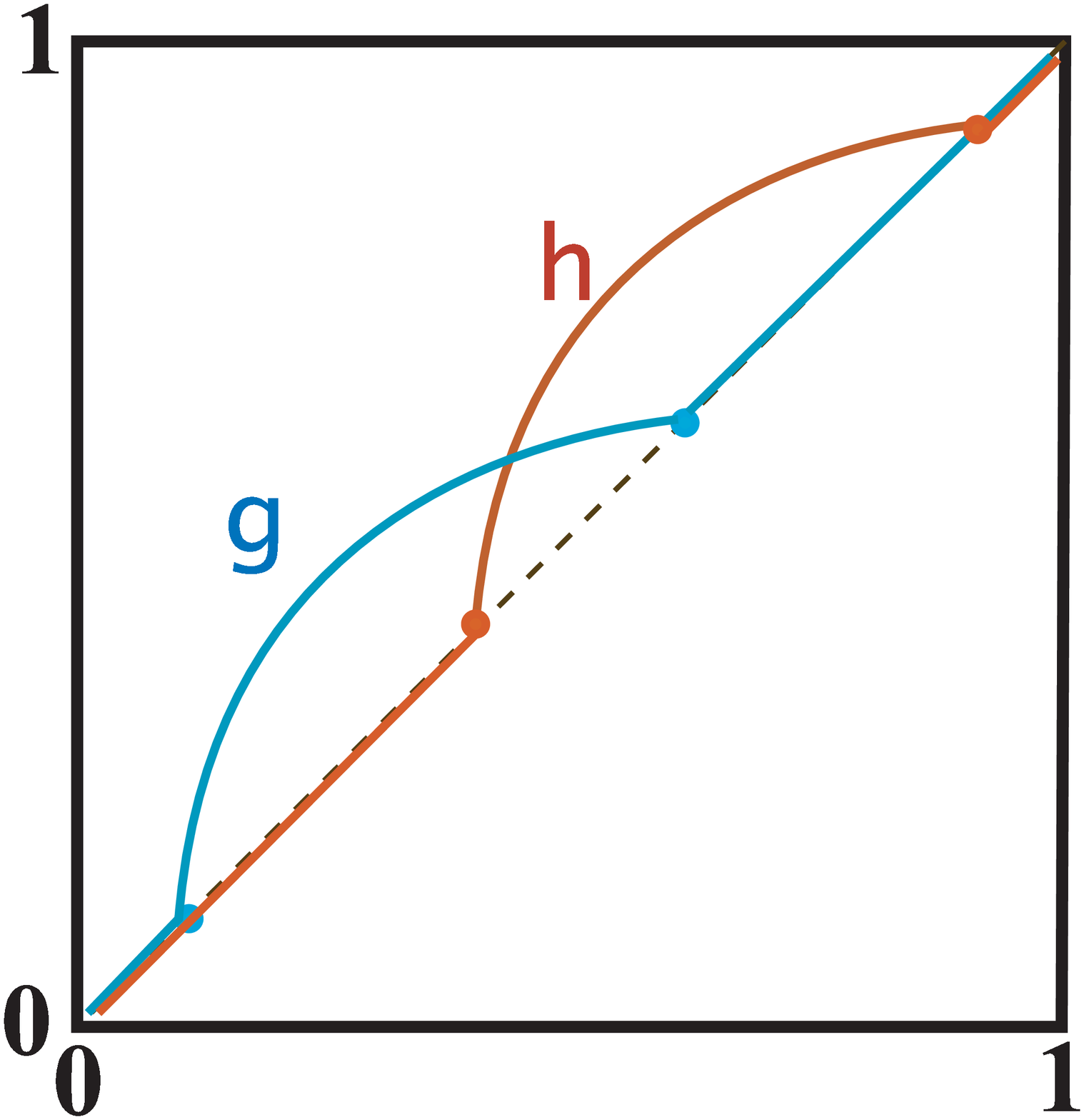}
\hspace{.25in}
\includegraphics[height=100pt,width=100pt]{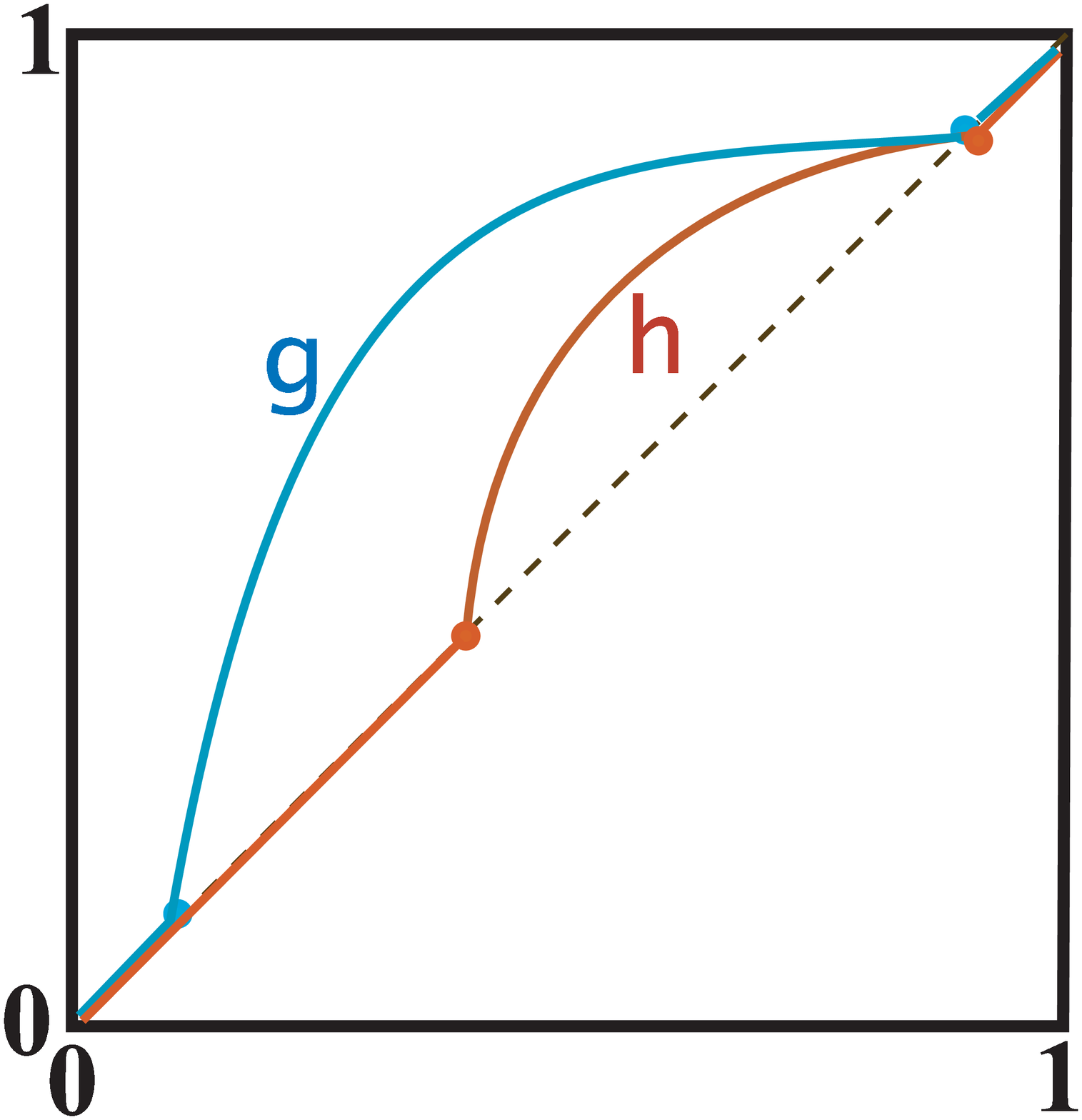}
\hspace{.25in}
\includegraphics[height=100pt,width=100pt]{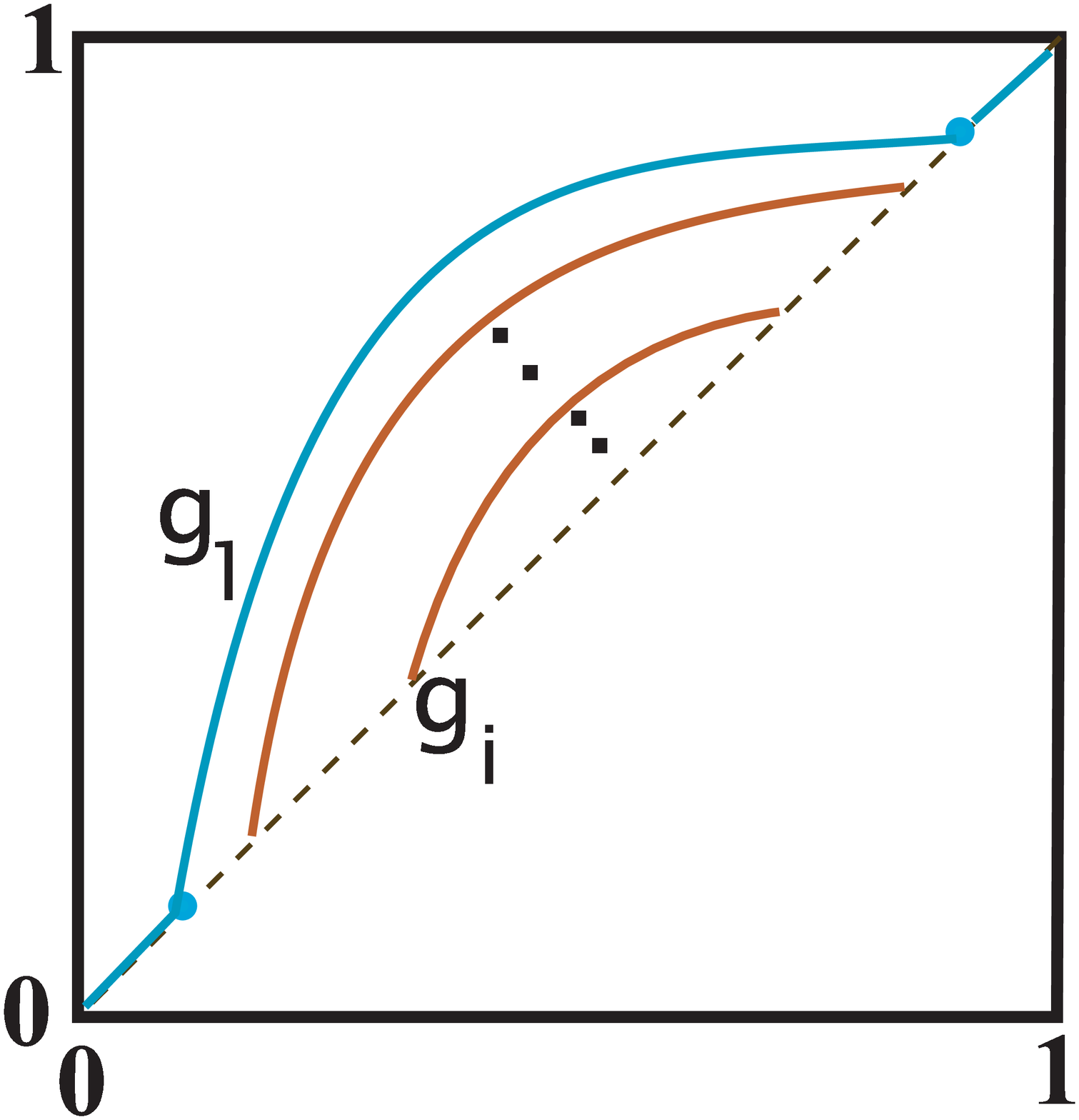}
\caption{Transition chain, \ones, and tower of infinite height}
\end{figure}
\end{center}

\vspace{-.25 in}

\smallskip

\noindent{\bf Theorem~\ref{main}.}
\emph{Let $G < \ploi$ be generated by a finite set $X$.
The following are equivalent.
\begin{enumerate}
\item $G$ is not soluble.
\item $G$ admits a transition chain.
\item $G$ admits a \ones.
\item $G$ admits a tower of infinite height.
\end{enumerate}
Moreover, 
if $G$ is soluble,
the derived length of $G$ is less than or equal to the
cardinality 
of the set $\breaks{X}$
of breakpoints of elements of $X$.
}

\smallskip

Our first application of Theorem~\ref{main}
is a verification of the conjecture discussed
in the first paragraph of this introduction.
Note that since a soluble group cannot contain a
non-soluble subgroup, a subgroup $G \le \ploi$
containing a copy of $B$ is also non-soluble.
Theorem 1.4 of \cite{bpnsc} 
(restated in Theorem~\ref{thm:tchainB} in Section~\ref{ss:B})
states that if $G$ admits a transition chain, then it admits 
an embedded copy of Brin's group $B$, and so we obtain the following
corollary 
resolving the conjecture.

\begin{corollary}\label{thm:fgnonsol-B}
Let $G$ be a finitely generated subgroup of $\ploi$.
Then $G$ is non-soluble if and only if Brin's group $B$
embeds in $G$.
\end{corollary}
 
Thus $B$ not only contains every soluble
subgroup of $\ploi$, but $B$ is also contained
in every finitely generated non-soluble subgroup
of $\ploi$.

In Section~\ref{sec:algorithm} we apply
Theorem~\ref{main} to 
develop a solution to 
the \emph{soluble subgroup recognition problem (SSRP)}.
Given a group $G$ with a finite generating set $Y$, 
the SSRP
asks whether there is an algorithm that,
upon input of a finite
set $X$ of words over $Y^{\pm 1}$,
can determine whether or not the subgroup $\langle X \rangle$
of $G$ generated by $X$ is a soluble group.
For a subgroup $G < \ploi$,
we allow the 
the finite 
list $X$ of elements input into our procedure
to have formats other than
a list of words over a finite generating set of $G$;
for example,
if $G$ is R.~Thompson's group $F$, in
which all breakpoints are 2-adic rational
numbers and all slopes are powers of 2, we may
input an element as an integer list of numerators and
denominators of the breakpoints and slopes
of the homeomorphism.
At the same time, we restrict our consideration to
\emph{computable} subgroups of $\ploi$,
in which several basic operations
can be implemented (see p.~\pageref{def:comp} for the
full definition).
Examples of requirements for a group 
$C\leq \ploi$ to be computable are that
the breakpoints and endpoints of 
components of support of elements can be computed
and compared,
and that given a finite 
collection of slopes of affine components of 
graphs of functions in $C$, a computer
can determine whether
the multiplicative subgroup of $\Rbb_+^*$  
generated by these slopes is discrete (has a lower bound on the 
distance from the identity 1 for all non-identity elements).
In particular, we note that 
the word problem is solvable in finitely generated
computable subgroups of $\ploi$, and that
R.~Thompson's group $F$ is a computable subgroup.

\smallskip

\noindent{\bf Theorem~\ref{thm:algorithm}.}
\emph{Let $C$ be a 
computable subgroup of $\ploi$.
The soluble subgroup recognition problem 
is solvable for $C$; that is, 
there is an algorithm which, upon input of
a finite subset $X$ of $C$, can determine
whether or not the
subgroup $\langle X \rangle$ generated by $X$ is 
a soluble group.  
Moreover, in the case
that the group $\langle X \rangle$ is soluble, the algorithm also 
determines its derived length.}

\smallskip

The proof of Theorem~\ref{thm:algorithm}
uses the concept of controllers for
chainless groups developed by the first author in~\cite{bpasc};
we discuss background on this topic
in Section~\ref{subsec:controllers-chainless}.
The main step of the procedure takes a finite collection 
$\{g_1,g_2,\ldots, g_k\}$ 
of elements generating a chainless group $G<\ploi$, 
with all of these 
generating elements having $(a,b)$ as a common component 
of support, and generates a single (``controlling'') 
element $c$ with $(a,b)$ as a component of support, and 
multiple other elements $\{h_1, h_2, \ldots, h_k\}$ with 
supports whose closure is contained in $(a,b)$, such that 
$G = \langle h_1,h_2, \ldots, h_k, c\rangle$.  
The set of elements $\{h_1,h_2,\ldots,h_k\}$ might 
then share components of support again inside $(a,b)$, 
when one continues to induct with the procedure on these 
new components of common support.  This process thus 
creates descending towers as in the 
graph on the right in Figure~\ref{fig:threeSits}, 
possibly descending forever; Theorem~\ref{main}
is used to determine when sufficient
information has been found to terminate this
procedure.

The proof of Theorem~\ref{thm:algorithm} also
shows that the {\em membership decision problem (MDP)} 
is solvable for some finitely generated soluble subgroups of 
computable subgroups of $\ploi$.
Given a group $G$ with a finite generating set $Y$
and a subgroup $H$ of $G$, 
the MDP for the subgroup $H$
asks whether there is an algorithm that,
upon input of
any word $w$ over $Y^{\pm 1}$,
can determine whether or not $w$ lies in the
subgroup $H$.
As with our SSRP algorithm, we allow inputs to our
MDP procedure to take forms other than
words over a finite generating set for $G$.
A set $X \subset \ploi$ is a
\emph{set of \nice}\label{def:nice} if $X$
satisfies the following properties:
\begin{itemize}
\item[(Z0)] Each element $h$ of $X$ admits exactly
one component of support, which we denote by $A_h$.
(That is, the graph of $h$ has ``one bump''.)
\item[(Z1)] No pair of elements of $X$ forms a transition
chain or a \ones.
\item[(Z2)] If $h,h' \in X$ and $h \neq h'$, then
$A_h \neq A_{h'}$.
\item[(Z3)] For each $h \in X$, there is an $r_h \in A_h$
such that for every $h' \in X$ with $A_{h'} \subsetneq A_h$, 
the containment $A_{h'} \subseteq (r_h,r_h \cdot h)$
also holds.
(That is, $(r_h,r_h \cdot h)$ is a fundamental
domain for the conjugation action 
by powers of $h$.)
\end{itemize}
In Lemma~\ref{lem:towertree} we show that
a group $H=\langle X \rangle$
generated by a finite set $X$ of \nice\ is a soluble group,
contained in the smallest class of groups
that includes the trivial group and is closed
under wreath products with $\Zbb$ and
finite direct sums.

\smallskip

\noindent{\bf Corollary~\ref{thm:membership}.}
\emph{Let $C$ be a 
computable subgroup of $\ploi$.
Let $H$ be a subgroup of $C$
generated by a finite set of \nice.
Then the membership decision problem 
is solvable for $H$; that is,
there is an algorithm which, upon input of
an element $w$ of $C$, can determine
whether $w \in H$.}

\smallskip

The results of this paper can be seen
as part of a larger family of 
results that researchers have
obtained by studying subgroups of 
$\ploi$ (or of R.~Thompson's group $F$) through 
a close attention to the dynamical properties of 
the action of the subgroup on the unit interval.  
Other results in this family include:

\begin{itemize}
\item Any subgroup of $\ploi$ satisfying the 
  Ubiquity Condition of Brin from~\cite{brinU}
  (restated in Theorem~\ref{thm:brinU} in 
   Section~\ref{subsec:controllers-chainless}),
   has an embedded copy of R.~Thompson's group $F$.
   (An example of this condition is given
  by the subgroup generated by the two elements
  in the \ones\ of 
  the centre graph of Figure~\ref{fig:threeSits}.)
\item  Any subgroup of $\ploi$ containing
  elements that yield a ``tower of height $i$''
  (depicted in rightmost graph of Figure~\ref{fig:threeSits}
  and discussed in more detail in 
  Section~\ref{subsec:orbitals-towers-chains})
  has derived length at least $i$~\cite{bpgsc,bpasc}
  (restated in Theorem~\ref{thm:towerSolveClass} below). 
\item The group $\ploi$ has no embedded non-abelian 
  free groups \cite{brinSquirePLR}.
\item Any non-abelian subgroup of R.~Thompson's group $F$ 
  contains an embedded copy of 
  $\Zbb\wr\Zbb$~\cite{gubaSapirDiagramGroups}.
\item Any subgroup $G$ of $\ploi$ is soluble if and only if 
  $G$ embeds as a subgroup of 
  $W_n =(\ldots((\Zbb\wr\Zbb)\wr \Zbb)\wr \ldots \Zbb)\wr \Zbb$ 
  for some $n\in\Nbb$ where $W_n$ has $n$ copies of $\Zbb$ 
  appearing in the iterated wreath product \cite{bpasc}.
\item Any non-soluble subgroup of $\ploi$ contains an embedded 
  copy of $W=\bigoplus_{n\in\Nbb} W_n$~\cite{bpnsc}.
\end{itemize}

Before proceeding to the proofs of
Theorems~\ref{main} and~\ref{thm:algorithm}
in Sections~\ref{sec:main} and~\ref{sec:algorithm},
we begin in Section~\ref{sec:ploi} with background,
notation, and definitions for the group $\ploi$.


\section{The group $\ploi$}\label{sec:ploi}


Here we give the basic definitions we require for 
discussing the group $\ploi$ and its elements.  
The results discussed in this section are 
first introduced in either \cite{brinU}, or 
later in \cite{bpgsc,bpasc,bpnsc}.


\subsection{Right actions, supports, slopes, and breaks}\label{subsec:actions-supports-breaks}


Throughout this paper we will use right action notation.  
In particular, if $x\in[0,1]$ and $g\in\ploi$, 
we write $xg$ for the image of $x$ under the map $g$.  
As is somewhat traditional (but not universal) 
for right actions, for elements $g,h\in\ploi$, 
and $S\subset[0,1]$ we set
\begin{align*}
Sg\defeq & \left\{sg\mid s\in S\right\}\\
\supp{g}\defeq &\left\{x\in[0,1]\mid xg\neq x\right\}\\
g^h\defeq & \,h^{-1}gh\\
[g,h]\defeq & \,g^{-1}h^{-1}gh
\end{align*} 
for the 
\emph{image of $S$ under the action of $g$},
the \emph{support of $g$},  
the \emph{conjugate of $g$ by $h$}, and the 
\emph{commutator of $g$ and $h$}, respectively.  
With this notation in place we have a 
standard lemma from permutation group theory, 
restated for elements of the group $\ploi$.

\begin{lemma}\label{lem:supports}
Let $g,h\in \ploi$.  Then
$\supp{g^h}=\supp{g}h.$
\end{lemma}

For a subgroup $G \le \ploi$, the associated 
\emph{slope group of $G$}, denoted $\Pi_G$, 
is the multiplicative subgroup of the positive real numbers 
generated by the slopes of affine components
of elements of $G$.

For $g\in\ploi$, $x\in(0,1)$, we say that 
$x$ is a \emph{breakpoint} 
of $g$ whenever $xg'$ does not exist 
(here, we are using $g'$ to denote the derivative of $g$).  
For a set $X\subset \ploi$ we denote by 
$\breaks{X}$ the set of breakpoints 
of the elements in $X$.  That is
$$
\breaks{X}\defeq\left\{x\in(0,1)\mid
   \exists g\in X, xg' \textrm{ does not exist}\right\}.
$$
We will slightly abuse this notation for a single element 
$g\in\ploi$, by setting 
$$
\breaks{g}\defeq\breaks{\left\{g\right\}}=
  \left\{x\in(0,1)\mid xg'\textrm{ does not exist}\right\}.
$$


 \subsection{Orbitals, towers, and transition chains}\label{subsec:orbitals-towers-chains}


We extend the definition of support to groups, 
so for a group $G\leq \ploi$, we set 
$$
\supp{G}\defeq\bigcup_{g\in G}\supp{g},
$$
noting that if $x\in\supp{G}$ then there is 
some $g\in G$ so that $xg\neq x$.  
As $\supp{G}$ is an open set, it can be written as a 
disjoint union of open intervals, each one of which 
is called an \emph{orbital of $G$}.  
That is, an orbital of $G$ is a connected component of the 
support of the action of $G$ on $[0,1]$.  
Note that if $A\neq B$ are two orbitals of $G$, 
then there is no element of $G$ which can move a point 
in $A$ to a point in $B$, which partly motivates our language
(as it means that each $G$-orbit is contained in some orbital of $G$).  
For an element $g\in \ploi$, a subset $A=(a,b)\subset[0,1]$ 
is an orbital of $\langle g\rangle$ if and only if
$A$ is a component of $\supp{g}$; in this case
we say $A$ is an \emph{orbital of $g$}.

A \emph{signed orbital} is a pair $((a,b),g)$
consisting of an open interval $(a,b)\subset [0,1]$ 
and an element $g\in\ploi$ such that $(a,b)$ is an orbital of $g$;
here $(a,b)$ is the orbital 
and $g$ is the \emph{signature}.  

A \emph{tower} is a set $\msc{T}$ 
of signed orbitals satisfying the property that
whenever  $((a,b),g)$ and $((c,d),h)$ are in $\msc{T}$, then
\begin{enumerate}
  \item   $(a,b)\subseteq (c,d)$ or $(c,d)\subseteq(a,b)$, and
  \item $(a,b)=(c,d)$ implies $g=h$.
\end{enumerate}
For a tower $\msc{T}$, the cardinality $|\msc{T}|$ is called 
the \emph{height of $\msc{T}$}.  
Given a group $G\leq \ploi$ and a tower $\msc{T}$, 
we say that $\msc{T}$ is \emph{associated} with $G$, or 
that $G$ \emph{admits} the tower $\msc{T}$, if all the signatures 
of the signed orbitals in $\msc{T}$ are elements of $G$. 
If $G \leq\ploi$ and $A$ is an open subinterval of the
unit interval $I$, the \emph{orbital depth} of $A$ in $G$
is the supremum of the heights of finite towers associated
with $G$ in which the smallest orbital has the form
$(A,g)$ for some $g \in G$.
If $G\leq\ploi$, we set the \emph{depth of $G$} to be the 
supremum of the heights of the towers in the full set of 
towers associated with $G$. 

The graph on the right in Figure~\ref{fig:threeSits} depicts a tower 
of infinite height.  
The ordering of indices, which appears 
inverted, favours the perspective of ``depth'' over ``height.''  
Reasons for this will become apparent in our construction 
proving Theorem~\ref{thm:algorithm}.

The main theorem of~\cite{bpgsc} is the following.

\begin{theorem}\label{thm:towerSolveClass}\cite[Theorem~1.1]{bpgsc}
Let $G\leq \ploi$ and $n\in\Nbb$.  
The group $G$ is soluble with derived length $n$ if and only if 
the depth of $G$ is $n$.
\end{theorem}


For $n\geq 2$, a \emph{transition chain of length $n$}
is a set
$$
\msc{C}=\left\{((a_i,b_i),g_i)\mid i\in\Nbb, 1\leq i\leq n\right\}
$$ 
of signed orbitals satisfying the property that
$$
a_1<a_2<b_1<a_3<b_2<a_4<b_3<\cdots<b_n.
$$ 
  If $G\leq \ploi$ and for all signatures $g$ of signed orbitals 
  in $\msc{C}$, we have $g\in G$, then we say $\msc{C}$ is 
  \emph{associated} with $G$, and that $G$ \emph{admits} a transition chain of 
  length $n$.  
  Note that if $G$ admits a transition chain of 
  length $n$ for some $n \ge 2$, then $G$ admits transition chains 
  of length $m$ for all $m\in \{2,\ldots,n\}$.  
We say $G$ is \emph{chainless} if $G$ admits no transition chains.
(We note that
other papers allow $n=1$ in the definition of a transition chain;
in this paper we require $n \ge 2$ in
the definition above in order to
streamline the phrase ``admits a transition chain''
without having to include ``of length 2''.)


 Already in \cite{bpgsc} a rudimentary connection 
 between a group ${H\leq \ploi}$ admitting transition chains 
 and the depth of the group $H$ is observed.
 
\begin{lemma}[Lemma~1.4 of~\cite{bpgsc}]\label{lem:weak-tower}
  If $H$ is a subgroup of $\ploi$ and $H$ admits transition 
   chains, then $H$ admits infinite towers.
\end{lemma}
 
 In particular, such groups have infinite depth 
 (they are \emph{deep}), and so by Theorem~\ref{thm:towerSolveClass}
 they are not soluble.
  
 On the other hand, chainless groups also have very special 
 properties relating to their towers, and also to how their 
 element orbitals can intersect each other.  
 Before stating these results, we give a further refinement of 
 our definition of towers.
 
 A tower $\msc{T}$ is \emph{exemplary} if whenever 
 $(A,g), (B,h)\in \msc{T}$ with $A\subsetneq B=(a,b)$ then
\begin{enumerate}
  \item the orbitals of $g$ are disjoint from the ends of 
    the orbital $B$, and
  \item no orbital of $g$ in $B$ shares an end with $B$.
\end{enumerate}  
Another way to put this is that there is an $\epsilon>0$ 
so that for any orbital $C$ of $g$ we have $C\cap B\neq\emptyset$ 
implies $C\subset (a+\epsilon,b-\epsilon)$.

We can now express a useful lemma describing element support 
overlaps and towers in chainless groups.  This lemma represents 
properties (2) and (3) of Lemma 2.7 in the paper~\cite{bpnsc}.

\begin{lemma}\label{lem:chainlessProps}\cite[Lemma~2.7]{bpnsc}
Suppose that $G\leq\ploi$ is a chainless group.
\begin{enumerate}
  \item If $\msc{T}$ is a tower associated with $G$, 
    then $\msc{T}$ is exemplary.
  \item If $f,g\in G$, $A$ is an orbital of $f$, 
    $B$ is an orbital of $g$, and $A\cap B\neq \emptyset$, 
    then exactly one of the following three statements holds:
  \begin{enumerate}
  \item $A=B$ and $A$ is an orbital of $\langle f,g\rangle$,
  \item $\overline{A}\subset B$, $A\cap Ag = \emptyset$, 
    and $B$ is an orbital of $fg$, $gf$, and of 
    $\langle f,g\rangle$, or
  \item $\overline{B}\subset A$, $B\cap Bf = \emptyset$, 
    and $A$ is an orbital of $fg$, $gf$, and of 
    $\langle f,g\rangle$.
\end{enumerate}
\end{enumerate}
\end{lemma}

To simplify notation later,
we say that $G \leq \ploi$ admits a \emph{\bado}
if there exists a pair of signed orbitals $(A,f)$
and $(B,g)$ associated to $G$ such that $A \cap B \neq \emptyset$,
$A \neq B$, $\overline{A}\not\subset B$,
and  $\overline{B}\not\subset A$; that is,
for these intersecting orbitals, the closure of
one of the orbitals contains exactly
one endpoint of the other orbital.  
The group $G$ admits a \bado\ if and only if $G$ admits
either a transition chain or a \ones.
It follows from Lemma~\ref{lem:chainlessProps} that any subgroup 
of $\ploi$ admitting a \ones\ admits a transition chain (we will see 
in Theorem~\ref{main} that the converse is also true).
Corollary~\ref{cor:badoverlap} then 
follows immediately from Theorem~\ref{thm:towerSolveClass}
and Lemma~\ref{lem:weak-tower}

\begin{corollary}\label{cor:badoverlap}
If $G\leq\ploi$ and $G$ admits a \bado,
then $G$ is not soluble.
\end{corollary}





 \subsection{The split group and one-bump factors}\label{ss:split}


Let $G \leq \ploi$.  Given an element $f \in G$,
let $A_1,...,A_k$ be the orbitals of $f$.
For all $1 \le i \le k$, let $f_i$ be the
element of $\ploi$ defined by $f_i|_{A_i} = f|_{A_i}$,
and $xf_i=x$ for all $x \in I \setminus A_i$.
Each function $f_i$ has precisely one
component of support, the functions
$f_i$ commute with each other, and
$f=f_1 \cdots f_k$.
We call these functions $f_i$ the
\emph{one-bump factors} of $f$, and
we refer to the signed orbitals $(A_i,f_i)$
as the \emph{factor signed orbitals}
associated to $f$.

The \emph{split group $S(G)$} associated to
the group $G \leq \ploi$, introduced in \cite{bpasc},
is  the group generated by the one-bump factors 
of all of the elements of $G$.
Note that $G$ is a subgroup of $S(G)$, and that
$S(S(G))$ might not be the same group as $S(G)$ in general. 
Whenever $G$ is a subgroup of another group $H$,
then $S(G) \le S(H)$.  
It is immediate from the chain rule that the slope 
group of $S(G)$ is also the slope group of $G$;
that is, $\Pi_{S(G)}=\Pi_{G}$. 

We record the following result of the first author from~\cite{bpasc}
for use in Section~\ref{sec:algorithm}.

\begin{theorem}\label{thm:splitderlength}~\cite[Cor.~4.6]{bpasc} 
Suppose that $G$ is a subgroup of $\ploi$.  The derived
length of $G$ equals the derived length of $S(G)$.
\end{theorem}


 \subsection{The group $B$}\label{ss:B}


 Brin's group $B$ is introduced in a general form as 
 $G_1$ in Section~5 of~\cite{brinEA}.  A presentation of 
 $B$ is given by
 $$
 B=\langle \{w_i \mid i \in \Zbb\},s\mid w_i^s=w_{i+1}, [w_i^{w_k^m},w_j]= 1 
  \ (i<k, j<k, m\in \Zbb\setminus\{0\})\rangle.
 $$  
 Note that this is an ascending HNN extension of the 
 group generated by the $w_i$ using stable letter $s$.
 Applying Tietze transformations
 shows that $B$ is generated by the elements $s$ and $w_0$.
 These two elements are illustrated in Figure~\ref{fig-B};
 for a detailed piecewise definition of these two
 homeomorphisms, see~\cite{bpnsc}.
 
 \psfrag{s}{$s$}
 \psfrag{w0}{$w_0$}
 \begin{center}
 \begin{figure}
 \includegraphics[height=200pt, width=200pt]{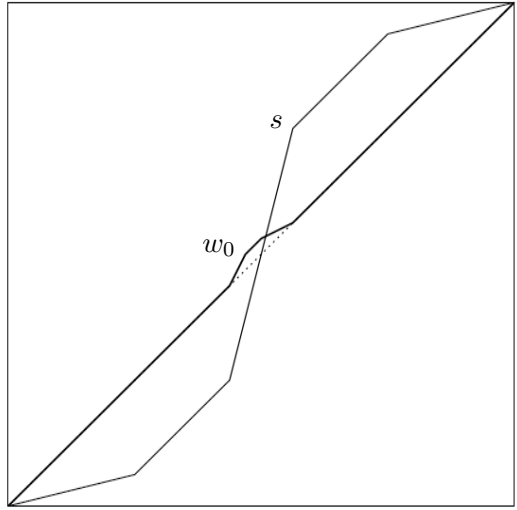}
 \caption{Generators of $B$\label{fig-B}}
 \end{figure}
 \end{center}
 \vspace{-.23 in}
 The ``s-curve'' $s$ acts by conjugation taking the 
 generator $w_0$  to an element $w_1$ with much larger support, 
 and such that $\supp{w_0}\cap \supp{w_0^{w_1}}=\emptyset$. 
 Thus these two elements commute.  
 Dynamical arguments using the definitions of the 
 generators $s$ and $w_0$ can then be made to verify 
 all of the relations of our presentation of $B$.

 The following is one of the main theorems of \cite{bpnsc}.
 
\begin{theorem}[Theorem~1.4 of~\cite{bpnsc}]\label{thm:tchainB}
 If $G\leq \ploi$ admits a transition chain, then $B$ embeds in $G$.
\end{theorem}


  \subsection{Controllers for chainless groups}\label{subsec:controllers-chainless}


  The material in this subsection is used in our proof of 
 Theorem \ref{thm:algorithm}, and relies primarily on 
 Sections~3.3 and~4.2 of~\cite{bpasc}.  
 The key motivation is to understand what must happen 
 in the absence of Brin's Ubiquity condition.
  
  Suppose $a<d\in [0,1]$.  If $g\in\ploi$ has an orbital 
 of the form $(a,b)$ or $(c,d)$  where 
 $a< b\leq d$ and $a\leq c<d$ then we say 
 \emph{$g$ realises an end of $(a,d)$}.  
 If there are both $b$ and $c$ so that $a<b\leq c<d$ and 
 $g$ has orbitals $(a,b)$ and $(c,d)$ then we say 
 \emph{$g$ realises both ends of $(a,d)$}.  
 Finally, if $g$ has orbital $(a,d)$, then we say 
 \emph {$g$ realises $(a,d)$}. 
 (and in this last case we  
 also say that $g$ realises both ends of $(a,d)$).
  
\begin{theorem}[Brin's Ubiquity Theorem~\cite{brinU}]\label{thm:brinU} 
 If a group $H\leq\ploi$ contains an 
 element that realises exactly one end of 
 an orbital of $H$, 
 then $H$ contains a subgroup isomorphic to 
 R.~Thompson's group $F$.
\end{theorem}

Let $H$ be a subgroup of $\ploi$ with orbital $A$. 
If there is an element $h\in H$ such that $h$ realises 
one end of $A$ but not the other, then we say 
\emph{$A$ is imbalanced for $H$}.  
On the other hand, if whenever $h\in H$ realises 
one end of $A$, then $h$ realises the other,  
we say \emph{$A$ is balanced for $H$}.  
We say \emph{$H$ is balanced} if for every
$G\leq H$ and every orbital $A$ of $G$, 
the orbital $A$ is balanced for $G$.  
A group which is balanced will have, 
by definition, no orbital which satisfies 
Brin's Ubiquity condition.

Using this viewpoint, 
Lemma~\ref{lem:chainlessProps} gives rise to the following. 

\begin{corollary}\label{cor:chainlessIsBalanced}
If $H\leq \ploi$ is chainless then $H$ is balanced.
\end{corollary}

\begin{proof}
Suppose to the contrary that $H$ is chainless but
not balanced.  Then there is a subgroup $G$ of $H$,
an orbital $(a,b)$ of $G$, an element $g$ of $G$,
and an orbital $(c,d)$ of $g$ 
such that $a \le c$, $d \le b$, and
exactly one of the equations $a=c$ or $b=d$ holds.
Suppose that $a=c$ (the proof in the other case is
similar).  Then $d \in (a,b)$ so there is another
element $h$ of $G$ with $d \in \supp{h}$;
let $B$ be the orbital of $h$ containing $d$.
Then the signed orbitals $((c,d),g)$ and $(B,h)$
form a \bado, and so Lemma~\ref{lem:chainlessProps}
shows that $G$, and hence $H$, is not chainless,
giving the required contradiction. 
\end{proof}

Corollary~\ref{cor:chainlessIsBalanced} together with Lemma~3.12 of~\cite{bpasc}, 
give us the following useful consequence. 

\begin{lemma}[Controller Existence Lemma]\label{lem:controllers}
Suppose that $H\leq\ploi$ is a chainless group
 with a single orbital $A$, and suppose that there is 
 an element $h\in H$ which realises one end of $A$.  
 Let $\mathring{H}_A$ represent the subgroup of 
 $H$ which consists of all elements $g$ of $H$ 
 for which there is a neighbourhood in $A$ of the 
 ends of $A$ upon which $g$ acts as the identity. 
 Then there is an element $c\in H$ so that 
 $H=\langle c,\mathring{H}_A\rangle$, where 
 $c$ realises the orbital $A$ of $H$.
\end{lemma}

Following the language and discussion of 
Section~3 of~\cite{bpasc}, we call the element 
$c\in H$ of Lemma \ref{lem:controllers} a 
\emph{controller of $H$ over $A$}, noting that 
$c$ is not unique among the controllers of 
$H$ over $A$, but that each such controller 
agrees with $c$ or $c^{-1}$ over some neighbourhood 
of the ends of $A$.  
Now, given $h\in H$, 
there is a unique integer $k$ and 
$\mathring{h}\in \mathring{H}_A$ so that 
$h=c^k\mathring{h}$ with respect to our 
choice of controller $c$ for the orbital $A$.


\section{Transition chains in finitely generated non-soluble subgroups}\label{sec:main}


The key to understanding why finitely generated nonsoluble 
subgroups of $\ploi$ always admit transition
chains turns out to be a fact about orbits of breakpoints.

\begin{lemma}\label{orbits}  
If $G<\ploi$ is finitely generated with a finite generating set $X$, 
 then the set 
of all breakpoints of $G$ has finitely many orbits under 
the action of $G$.  Moreover, this number of orbits is 
bounded above by the cardinality of the set 
of breakpoints of the generating set $X$.
\end{lemma}

\begin{proof}
Let $G<\ploi$ be finitely generated by $X = \{a_1,a_2,\ldots,a_k\}$.  
We will show that every breakpoint of $G$ is in the same 
orbit as some breakpoint of one of the generators $a_i$.

Let $x\in (0,1)$ be a breakpoint of some $g\in G$, and 
write $g = \alpha_1\alpha_2\ldots \alpha_n$, where for
each $\alpha_i$, either $\alpha_i\in X$ or 
$\alpha_i^{-1}\in X$.
Let $j$ be maximal such that $\alpha_1\alpha_2\ldots \alpha_j$ 
has constant slope on some interval around $x$. 
Since $x$ is a breakpoint of $g$, we must have $j<n$.  

Let $y = x\cdot \alpha_1 \alpha_2 \ldots \alpha_j$.
By maximality of $j$, $y$ must be a breakpoint of $\alpha_{j+1}$.
We have $x = y\cdot \alpha_j^{-1} \alpha_{j-1}^{-1}\ldots \alpha_1^{-1}$,
and so $x$ is in the $G$-orbit of $y$.
If $\alpha_{j+1}\in X$ we are done.
Otherwise, if $\alpha_{j+1}^{-1} \in X$ we recall that
the breakpoints of $\alpha_{j+1}^{-1}$ are the images of the
breakpoints of $\alpha_{j+1}$ under the map $\alpha_{j+1}$.
Hence every breakpoint of $G$ is in the same $G$-orbit 
as one of the (finitely many) breakpoints of elements of 
the generating set $X$, establishing the lemma.
\end{proof}

We can now prove our main theorem.

\begin{theorem}\label{main}
Let $G < \ploi$ be generated by a finite set $X$.
The following are equivalent.
\begin{enumerate}
\item $G$ is not soluble.
\item $G$ admits a transition chain.
\item $G$ admits a \ones.
\item $G$ admits a tower of infinite height.
\end{enumerate}
Moreover, 
if $G$ is soluble,
the derived length of $G$ is less than or equal to the
cardinality 
of the set $\breaks{X}$
of breakpoints of elements of $X$.
\end{theorem}

\begin{proof}
Let $G=\langle X \rangle$ 
be a finitely generated 
subgroup of $\ploi$.
The implication (2) $\Rightarrow$ (4)
is Lemma~\ref{lem:weak-tower},
and the implication (4) $\Rightarrow$ (1) follows
immediately from Theorem~\ref{thm:towerSolveClass}.
Lemma~\ref{lem:chainlessProps} shows that
(3) $\Rightarrow$ (2).

Next we show that (2)  $\Rightarrow$ (3).
Suppose that $G$ admits a transition chain 
$\{((a,b),f),((c,d),g)\}$, where $a<c<b<d$.  

If $a$ is not in the support of $g$, 
then $f^g$ admits an orbital $(a,e)$ with $e\neq b$, 
and hence the pair of signed orbitals $\{((a,b),f), ((a,e),f^g)\}$ 
represents a \ones\ for $G$.

Similarly, if $d$ is not in the support of $f$, 
then there is an $e\neq d$ such that the pair 
$\{((c,d),g), ((e,d),g^f\}$  represents a \ones\ for $G$.

We extend this endpoint-support argument 
to the left and the right until we run out of orbitals 
of $f$ or of $g$.  Eventually we must fail to have 
an end of one of these signed orbitals in the support 
of an orbital of the other element, and therefore we can 
find a \ones\ where the signatures are either the elements 
$f$ and $f^g$ or the elements $g$ and $g^f$. 

Finally, we show both (1) $\Rightarrow$ (2) and
the claim on derived length.
Suppose that $G$ chainless. 
Let $n\in \Nbb$ so that $n-1$ is equal to the 
number of $G$-orbits of the set of breakpoints of $G$
(a finite number by Lemma~\ref{orbits}).  
Suppose that $G$ is either non-soluble or 
of derived length $z$, for some $z\geq n$.

We note in passing that we may assume $n>1$, since if $n=1$ then $G$ 
has no breakpoints, and so $G=\left\{1\right\}$ and hence the 
derived length of $G$ is $0$ which does not exceed the
 number $0$ of $G$-orbits in the set $\msc{B}_G$
of breakpoints of elements of $G$.  

Theorem~\ref{thm:towerSolveClass} implies that $G$ admits a tower 
\[
\msc{T}=\left\{(A_1,g_1),(A_2,g_2),\ldots,(A_n,g_n)\right\}
\] 
of height $n$,
which by Lemma~\ref{lem:chainlessProps} is exemplary,
and hence we may assume that the signed orbitals of 
$\msc{T}$ are indexed in such a fashion that 
for all indices $i<n$ we have $\overline{A}_{i+1}\subset A_{i}$.  
(In the case that $G$ is non-soluble, 
$G$ admits towers of arbitrary height by 
Theorem~\ref{thm:towerSolveClass}.)

The endpoints of $A_i$ might not be 
breakpoints of $g_i$; that is,
$g_i$ may have a disjoint orbital with the same
endpoint and the same slope for $g_i$ in a neighbourhood
of that endpoint.  To take this into account, we widen the
interval that we consider, as follows.
There is a maximal $k_i\in \Nbb$ 
such that there is an ordered tuple $X_i$ of signed orbitals 
\[
X_{i}=((A_{i1},g_i),(A_{i2},g_i),\ldots,(A_{ik_i},g_i)),
\]
where we write $A_{ij} = (a_{ij},b_{ij})$, satisfying the
properties that for each 
index $j<k_i$ we have $b_{ij}=a_{i(j+1)}$ and there 
is an index $m_i$ with $A_i=A_{im_i}$.  

Since $G$ is chainless, Lemma~\ref{lem:chainlessProps}
shows that $G$ does not admit {\bado}s.
From the fact that each $A_{ij}$ shares an end
with each of its `neighbours' in $X_i$, and 
$\overline{A}_{im_i} = \overline{A}_i\subset A_{i-1}$ for $i>1$,
we deduce that for $1<i<n$ and $1\leq j\leq k_{i}$ we have
$\overline{A}_{ij}\subset A_{i-1}$
since otherwise $G$ would admit a \bado.

Now, for each index $1\leq i<n$, let 
$c_i= a_{(i+1)1}$ and 
$d_i= b_{(i+1)k_{i+1}}$,
and let
$a_i= a_{im_i}$ and 
$b_i= b_{im_i}$.
Furthermore, set $c_n$ to be some 
breakpoint of $g_n$ in $A_n$ (such must exist since $g_n$ 
cannot be affine over $A_n$).  We do not define $d_n$.
If $1\leq i<n$ it is now the case (by the maximality of 
$k_{i+1}$) that $c_i$ and $d_i$ are 
breakpoints of the element $g_{i+1}$, and that 
$\overline {(c_i,d_i)}\subset A_i$.



As $n$ is larger than the number of orbits of 
breakpoints of $G$ under the action of $G$, there are 
indices $r<s$ so that $c_r$ and $c_s$ are in the same 
$G$-orbit. 
Hence there is an element $g\in G$ such that $c_r\cdot g=c_s\in A_s$.  
In particular, by Lemma~\ref{lem:supports} and 
the nonexistence of {\bado}s
we see that for each index $j$
the interval $A_{(r+1)j}\cdot g$ is an orbital 
of $g_{r+1}^g$ with  
closure properly contained in $A_s$ away from the ends 
of $A_s$.  In particular, we have $\overline{(c_r,d_r)}\cdot g \subset A_s$.  The above implies the following chain of relationships. 
\[c_r\leq a_{r+1} \le a_s < c_r\cdot g = c_s < d_r\cdot g < b_s \le b_{r+1}\leq d_r.\]
This means that $g$ moves $c_r$ to the right across $a_s$, while
also moving $d_r$ to the left across $b_s$.
However, any given orbital of $g$ has all of its points moved in 
the same direction by $g$, so $g$ must have at least two distinct 
orbitals, one orbital $(x,y)$ containing $a_s$ and another 
containing $b_s$.  
Consequently, we have that $x<a_s<y<b_s$ and 
so $\{((x,y),g),((a_s,b_s),g_s)\}$ is a transition chain of length two 
for $G$.
Since $G$ is chainless, this gives a contradiction, so we can conclude 
that $G$ is indeed soluble with derived length $z$ less than or equal 
to the number of $G$-orbits of the set of breakpoints of $G$.
Lemma~\ref{orbits} completes the proof.
\end{proof}

Note that the hypothesis that $G$ is finitely generated is not required 
for the equivalence of (2) and (3), and that these two conditions could 
be replaced by the single condition: ``$G$ admits a \bado".


\section{An algorithm to detect solubility}\label{sec:algorithm}


The goal of this section is to use
Theorem~\ref{main}
and the concept of controllers from 
Section~\ref{subsec:controllers-chainless}
to construct the algorithm
to solve the soluble subgroup recognition problem for the
proof of Theorem~\ref{thm:algorithm}.

Let $C\leq \ploi$.  In order to input a finite
list of elements of $C$ into our procedure,
we need to be able to write these elements
with some sort of data structure; for example,
if $C$ is R.~Thompson's group $F$, we may
input an element as a list of numerators and
denominators of the breakpoints and slopes
of the homeomorphism, since all of these are
rational numbers, but we may instead input the 
element as a word over a finite generating
set for $F$, as a tree pair diagram, or as any
other construct that encodes this information.
For whatever structure is used, there
are several pieces of information we need
to be able to calculate from this data,
which we list in the following processes.
Some of these processes are required
to hold for the (potentially larger) split group $S(C)$ 
(see Section~\ref{ss:split}
for this construction).

{\flushleft{\it {\bf Processes:}}}
\begin{enumerate}
\item Given $g, h\in S(C)$ determine $gh$ and $g^{-1}$. \label{proc:gpops}
\item Given $g\in C$, determine its set of 
    breakpoints $\msc{B}_g$.\label{proc:bp}
\item Given $g \in S(C)$ and a breakpoint or orbital endpoint $x$ of $S(C)$, 
   compute  $x \cdot g$.\label{proc:bpeval}
\item Given two points $a,b \in [0,1]$ that occur either as 
  breakpoints or as orbital endpoints of elements of $S(C)$, 
  determine whether $a<b$, $a=b$, or $a>b$.\label{proc:bpcompare}
\item Given $g\in S(C)$, produce the finite tuple 
  $X_g=[A_1, A_2,\ldots,A_{k_g}]$ of all 
  orbitals of $g$, where each $A_i$ is stored as the ordered pair 
  $(a_i,b_i) = (\inf(A_i),\sup(A_i))$, and 
  $a_1<a_2<\ldots<a_{k_g}$ .\label{proc:tuple}
\item Given a signed orbital $(A,g)$ associated with
  $S(C)$, output the factor signed orbital $(A,h)$ satisfying
 $h|_A=g|_A$ and $\supp{h}=A$.\label{proc:split}
\item Given a signed orbital $((a,b),g)$ of $S(C)$, determine the 
  slopes $m_{ga} \defeq ag_+'$ and $m_{gb}\defeq bg_-'$ 
  of the affine components of the graph of $g$ over $(a,b)$ 
  near $a$ and $b$ respectively.\label{proc:endptslope}
\item Given two elements $m_1,m_2$ of the slope group $\Pi_C=\Pi_{S(C)}$ of $C$, 
  compute $m_1m_2$ and $m_1^{-1}$, and
  determine whether or not $m_1<m_2$, $m_1=m_2$, 
  or $m_1>m_2$.\label{proc:slopecompare}
\item 
  Given a finite set $Z=\{m_1,m_2,\ldots,m_k\}$ of 
  positive numbers in $\Pi_C=\Pi_{S(C)}$,
  determine if the multiplicative group 
  $\Pi_Z\defeq \langle m_1,m_2,\ldots,m_k\rangle\leq\Rbb^*_+$ 
  is discrete.
  If this group $\Pi_Z$ is discrete, further determine 
  integers $p_1$, $p_2$, $\ldots$, $p_k$ so that  
  $\Pi_{Z,s}\defeq m_1^{p_1}m_2^{p_2}\cdots m_k^{p_k}$ 
  is the least value in $\Pi_Z$ greater than one.\label{proc:euclid}
\end{enumerate}

We say that a subgroup $C\leq \ploi$ is a
\emph{computable group}\label{def:comp}
if  the elements of $C$
have representatives for which this list of processes
can be carried out by a computer.

 For a subgroup $C$ of $\ploi$, if the sets of breakpoints, 
 orbital endpoints and slopes of affine 
 components of graphs of elements of $C$ are sufficiently
 specialised sets of values, then
 these processes can be performed.
 We observe that all of the processes above can be carried 
 out for elements in R.~Thompson's group $F$ by a modern computer.
 Moreover, $F$ is equal to its own
 split group $S(F)$, and so these processes can be performed
 in $S(F)$.
 For the most complex process, namely
 Process~\ref{proc:euclid}, one uses a generalised Euclidean Algorithm 
 on the log base two values of the sets of slopes to determine 
 the integers $p_i$ 
 in this case.  Hence $F$ is computable.

 We also note that any subgroup of a computable group 
 is computable.
For the algorithm we provide below, we actually work 
in subgroups of the split group $S(G)$ of our original 
computable group $G$.
As a consequence the following corollary and lemma will be
applied several times.  Corollary~\ref{cor:subsplitder} follows
immediately from Theorems~\ref{thm:splitderlength}  
and~\ref{main}, and
the fact that whenever $H_1$ is a subgroup of a group $H_2$,
the derived length of $H_1$ is at most the derived
length of $H_2$.

\begin{corollary}\label{cor:subsplitder}
Let $G$ be a subgroup of $\ploi$, and let $H$ be a subgroup
of the split group $S(G)$ containing $G$.
\begin{enumerate}
\item The derived length of $G$ equals the derived length
of $H$.
\item If $H$ admits a \bado, then $G$ also admits a 
\bado\ and $G$ is not soluble.
\end{enumerate}
\end{corollary}




We will also apply the following lemma 
in the proof of Theorem~\ref{thm:algorithm},
in order to verify that our subgroups
remain inside $S(G)$.

\begin{lemma}\label{lem:nonsplitrep}
Let $G \leq \ploi$.  Suppose that
$(A_1,g_1),...,(A_q,g_q)$ are factor signed orbitals
associated to elements of $G$.
\begin{enumerate}
\item If $A_1 = \cdots =A_q$ and if $p_1,...,p_q$ are any
integers, then the one-bump factors of
$g_1^{p_1} \cdots g_q^{p_q}$ are also one-bump factors
of an element of $G$.
\item If $\overline{A_1} \subset A_2$, then the conjugate
$g_1^{g_2}$ is also a one-bump factor
of an element of $G$.
\end{enumerate}
\end{lemma}

\begin{proof}
Let $\hat g_1,...,\hat g_q$ be elements of $G$
such that $(A_i,g_i)$ is an associated
factor signed orbital of $\hat g_i$ for each $i$.

First suppose that $A_1 = \cdots =A_q$ and $p_1,...,p_q \in \Zbb$,
and let $c \defeq g_1^{p_1} \cdots g_q^{p_q}$.
Let $\hat c \defeq \hat g_1^{p_1} \cdots \hat g_q^{p_q}$.
Then since $g_i|_{A_1}=\hat g_i|_{A_1}$ for all $i$,
and each $g_i|_{A_i}$ is a homeomorphism of the interval $A_i$,
we have $c|_{A_1} = \hat c|_{A_1}$.  Hence the
one-bump factors of $c$ are exactly the
one-bump factors of $\hat c$ whose support is
contained in the interval $A_1$.

Next suppose that $\overline{A_1} \subset A_2$.
By Lemma~\ref{lem:supports}, the support of
the conjugate $g_1^{g_2}$ is the interval
$A_1\cdot g_2$.  Since $g_2$ acts as a homeomorphism
of the interval $A_2$ and fixes the rest of $I \setminus A_2$,
then $A_1 \cdot g_2 \subseteq A_2$.  
Similarly the conjugation action of $\hat g_2$
on $\hat g_1$ takes the signed orbital $(A_1,\hat g_1)$
to the signed orbital 
$(A_1\cdot \hat g_2,\hat g_1^{\hat g_2})$.
Since $\hat g_2|_{A_2}=g_2|_{A_2}$, then
$A_1\cdot \hat g_2=A_1\cdot g_2$ and
on this interval the functions $\hat g_1^{\hat g_2}$ and
$g_1^{g_2}$ agree.
Thus $g_1^{g_2}$ is a one-bump factor of $\hat g_1^{\hat g_2}$.
\end{proof}


While Corollary~\ref{cor:subsplitder} and Lemma~\ref{lem:nonsplitrep}
are used toward determining when the input group $G$
is not soluble, the following lemma will be used
toward determining when $G$ is soluble.

\begin{lemma}\label{lem:towertree}
Suppose that $H < \ploi$ is  generated by a 
finite set $Z$ of \nice, and let
$S_Z$  
be the set of signed orbitals associated to the
elements of $Z$.
Then $H$ is a soluble group, and the derived length
of $H$ is the largest height of a tower of signed
orbitals contained in the set $S_Z$.
\end{lemma}

\begin{proof}
Let $n$ be the largest height of a tower of
signed orbitals that are contained in $S_Z$;
we proceed by induction on $n$.

If $n$=0, then $S_Z$ and hence $Z$ is empty,
and $H$ is the trivial group, which is soluble of
derived length 0.
If $n=1$, then Properties Z1-Z2 
of the definition of a set of \nice\ (p.~\pageref{def:nice})
imply that the supports of the elements
of the generating set $Z$ are pairwise disjoint,
and so the elements of $Z$ commute.  Therefore
$H$ is abelian, and so $H$ is soluble with
derived length 1.

Now suppose that $n>1$ and the result is true for finite
sets satisfying Properties 
Z0-Z3 with maximum associated tower height at most $n-1$.
For each element $h \in Z$, let 
$A_h$ denote the support $\supp{h}$ (in 
the notation of Property Z0) and let
$\orbd(A_h)$ denote
the maximum height of a tower built from
elements of $S_Z$ such that $A_h$ is the smallest
orbital (that is, $A_h$ is contained in the supports of all
of the other signed orbitals in the tower).
Let $Y:=\{h \in Z \mid \orbd(A_h)=1\}$,
and for each $h \in Y$, let 
$P_h:=\{h' \in Z \mid A_{h'} \subsetneq A_h\}$.
Property Z2 implies the set $Y$ does not
contain two elements with the same support.
Then Property Z1 implies that the elements of $Y$
have disjoint support, and so 
$H$ is the
direct product of the subgroups $\langle h,P_h \rangle$
for $h \in Y$.  

Note that each subset
$P_h$ of $Z$ satisfies Properties Z0-Z3,
and its associated signed orbitals have
maximal tower height $n-1$, so by induction the group
$\langle P_h \rangle$ is a soluble group for each $h\in Y$,
and the derived length of $\langle P_h \rangle$
is the maximal height of a tower that can
be built from signed orbitals associated
to elements of $P_h$.

Now Property Z3 implies that 
for distinct
integers $j$ the groups
$\langle P_h \rangle^{h^j}$  have disjoint
support, and so the subgroup of $\langle h,P_h \rangle$
generated by these subgroups is the
direct product $\oplus_{j \in \Zbb} \langle P_h \rangle^{h^j}$.
The conjugation action of the group
$\Zbb = \langle h \rangle$ in this direct product
permutes the summands.  Hence the group
$\langle h,P_h \rangle$ is a wreath product
$\langle h,P_h \rangle = \langle P_h \rangle \wr \Zbb$.
Then $\langle h,P_h \rangle$ is again soluble.  Moreover,
the derived length of $\langle h,P_h \rangle$ is
one more than the derived length
of $\langle P_h \rangle$; that is, it is
the maximum height of a tower associated to
elements of the set $\{h\} \cup P_h$.
Since at least one $\langle P_h\rangle$ 
has derived length $n-1$, this implies that 
some $\langle P_h,h\rangle$ has derived length $n$.

Putting these results together, we have
$H = \oplus_{h \in Y} (\langle P_h \rangle \wr \Zbb)$
is a soluble group, with derived length $n$.
\end{proof}

We are now in position to prove Theorem~\ref{thm:algorithm}.

\begin{theorem}\label{thm:algorithm}
Let $C$ be a computable subgroup of $\ploi$.
The soluble subgroup recognition problem 
is solvable for $C$; that is,
there is an algorithm which, upon input of
a finite subset $X$ of $C$, can determine
whether or not the
subgroup $\langle X \rangle$ generated by $X$ is 
a soluble group.  
Moreover, in the case
that the group $\langle X \rangle$ is soluble, the algorithm also 
determines its derived length.
\end{theorem}

\begin{proof}

Suppose $C$ is a computable subgroup of $\ploi$, and
$f_1,f_2,\ldots,f_m$ are elements of $C$
input to the algorithm, where $m$ is a positive integer. 
Let $G\defeq\langle f_1,f_2,\ldots,f_m\rangle$.
Then $G \leq C$; hence 
$G$ is also computable.

Before giving the technical details of the
algorithm, we begin with an overview of
our procedure. 
In the algorithm below, we build the tree 
of towers (up to conjugation equivalence of towers) 
for the group $S(G)$.  
We apply a breadth-first-search 
to the tree of nested orbitals of these towers 
(successively moving left to right
through all orbitals at the least depth
before moving on to orbitals with greater depth),
looking for {\bado}s. 

In the steps of this algorithm we 
maintain a finite 
set $SO$ of signed orbitals of the split group $S(G)$.
For any collection $S'$ of signed orbitals, there
is an associated \emph{signature group}, 
denoted by $\siggp(S')$, which is the 
group generated by the signatures 
of the orbitals in $S'$.
At every step the set $SO$ and its associated
signature group will satisfy the
following properties:
\begin{itemize}
\item[$SO$.1] Each element of $SO$ has 
a signature that is a one-bump factor of an element of $G$,
and hence $\siggp(SO)$ is a subgroup of $S(G)$.
\item[$SO$.2] The signature group $\siggp(SO)$ contains $G$.
\end{itemize}


We will also maintain two disjoint sets
$\seen$ and $\unseen$ of orbitals
(representing the ``seen'' and ``unseen'' orbitals, respectively), 
whose union is the collection
of unsigned orbitals associated to the signed orbitals in $SO$.
Each orbital $O$ will arise from Process~\ref{proc:tuple},
and so will be stored by the algorithm in the format 
$(\inf(O),\sup(O))$; that is,
by storing the endpoints of the interval $O$.
The orbitals in $\seen$ will satisfy the properties
\begin{itemize}
\item[$\seen$.1] 
No pair of signed orbitals in $SO$ whose (unsigned)
orbitals lie in $\seen$ forms a \bado.
\item[$\seen$.2] For every orbital $A$ of $\seen$,
there is exactly one signed orbital in $SO$ with
$A$ as its support; we denote this signed orbital 
$\sigma_A=(A,h_A)$.
\item[$\seen$.3] For every orbital $A$ of $\seen$,
there is a point $r_A \in A$
such that for every $A' \in \seen$ with $A' \subsetneq A$, 
the containment $A' \subseteq (r_A,r_A \cdot h_A)$
also holds.
\end{itemize}
Note that these properties imply that the set
$Z:=\{h_A \mid A \in \seen\}$ of signatures associated
to the orbitals in $\seen$
satisfies conditions Z0-Z3 of 
the definition of a set of \nice,
but properties $\seen$.1-$\seen$.3 also include a partial
extension of Z0-Z3 to $SO$.

We further partition $\unseen$ into sets $\Top$
and $\Lower$, to keep track of the order
in which orbitals will be processed.
Some of the orbitals $O$ in $\unseen$ and all
of the orbitals in $\seen$
will be assigned an ``orbital depth value''
$\orbd(O)$, which is a lower bound on the
numerical value of the orbital depth of $O$ in the group $\siggp(SO)$;
in particular, $\orbd(O)=n$ will mean that the
algorithm has found an exemplary tower of
height $n$ associated with $\siggp(SO)$
with a signed orbital of the form $(O,g) \in SO$
at the bottom.



As our computation proceeds, new (signed or unsigned) orbitals 
will be added to $SO$ and $\unseen$, and in other 
steps element orbitals will move from $\unseen$ to $\seen$
or will be removed from $SO$ or $\unseen$.  
Our calculation will terminate either when the algorithm
detects either a \bado\ in the group $\siggp(SO)$ or
an orbital that is `too deep',
or else (soon) after
all orbitals have been removed from $\unseen$, so that
$\unseen = \emptyset$.  We will process the set 
$\unseen$ carefully, keeping track of
the height of towers that have been
found, so that we will be guaranteed that the 
algorithm will stop if it finds no {\bado}s.


Throughout the description of the algorithm
we also include proofs that the sets
$SO$ and $\seen$ have the properties $SO$.1,$SO$.2,
and $\seen$.1,$\seen$.2,$\seen$.3 respectively, as well
as other commentary adding information about
the steps along the way.
In order to distinguish between
steps of the algorithm and explanations of its validity, 
we number and indent the steps of the algorithm. 
The remaining bulk of the proof that the algorithm 
is valid is provided after all of the steps
have been described.

\bigskip\eject

\centerline{{\it \underline{Start of algorithm}}}
{\flushleft{\it \underline{Step $0$} (Setting up the algorithm):}}\\

\begin{itemize}
\item[0.1] {Let $SO$, $\unseen$, $\seen$, $\Top$, and $\Lower$ be empty sets.
Let $\maxDepth := 0$ and $\counter\defeq 0$.
}
\end{itemize}

(Note that $\seen$ satisfies properties $\seen$.1,
$\seen$.2, and $\seen$.3 here.)

\begin{itemize}
\item[0.2] {For each input element $f_i$:  
Apply Process~\ref{proc:tuple} to
compute the tuple 
$X_{f_i}=  [A_{i1},A_{i2},\ldots, A_{ik_i}]$ of orbitals of $f_i$.   
Next use Process~\ref{proc:split}
to compute the corresponding $k_i$ signed orbitals 
$(A_{ij},\tilde{f}_{ij})$ associated to the split group $S(G)$, where
the $\tilde{f}_{ij}$ are the one-bump 
functions associated to $f_i$
(that is, $\tilde{f}_{ij}$ equals $f_i$ over $A_{ij}$).  
Add the pairs $(A_{ij},\tilde{f}_{ij})$ to the set $SO$ 
and add the orbital parts $A_{ij}$ to $\unseen$. 
}
\end{itemize}

Note that since the set
$SO$ contains the set of factor signed orbitals of
the generating set $X\defeq \{f_1,...,f_m\}$ of $G$, 
this set $SO$ satisfies properties $SO$.1 and $SO$.2.

\begin{itemize}
\item[0.3] {Compute the value $n := |\breaks{X}|$ (the total 
number of breakpoints of elements of $X$)
using Processes~\ref{proc:bp} and~\ref{proc:bpcompare}. 
}
\end{itemize}


{\flushleft{\it \underline{Step $1$} (Building $\Top$ and 
  $\Lower$ from $\unseen$):}}\\


\begin{itemize}
\item[1.1] Check whether $\unseen=\emptyset$. If so, then 
terminate the algorithm and output 
``The group $G$ is soluble with derived length $\maxDepth$''.


\item[1.2]  Determine, using Process~\ref{proc:bpcompare},
whether or not $\unseen \cup \seen$, and therefore $SO$, contains a \bado.
If so, then terminate the algorithm and output 
``The group $G$ is not soluble.''


\item[1.3] For all $A$,$B$ in $\unseen$:
Using Process~\ref{proc:bpcompare}, determine whether
$\bar{A}\subset B$, and if so add $A$ to $\Lower$.  

\item[1.4]  Let $\Top \defeq$ the complement of $\Lower$ in $\unseen$.\\
Let $\counter \defeq \counter+1$.
\end{itemize}

The variable $\counter$ is used to record, for use in
Step 2.1, whether Step 1 has been performed more
than once; after Steps 1-3 are done, the
algorithm can loop back to Step 1 again.



{\flushleft{\it  \underline {Step $2$} (Processing the 
  orbitals in $Top$ to detect excessive depth):}}\\

\begin{itemize}
\item[2.1] If $\counter=1$, then for all orbitals $A \in \Top$,
assign  the value $\orbd(A) \defeq 1$.  Otherwise, if $\counter > 1$,
then for each orbital $A \in \Top$, assign the value
$\orbd(A) \defeq 1 +$ the number of orbitals in $\seen$
that contain $A$.
(This requires Process~\ref{proc:bpcompare}.)

\item[2.2]  Compute $\maxDepth := 
     \max(\{\orbd(A) \mid A \in \Top\} \cup \{\maxDepth\})$.
If $\maxDepth > n$, terminate the algorithm and output
``The group $G$ is not soluble.''
\end{itemize}

Note that since this maximum is taken with
the old value of $\maxDepth$ included,
successive occurrences of Step~2.2 cannot
decrease the value of $\maxDepth$.




{\flushleft{\it  \underline {Step $3$} (Processing the leftmost element of $\Top$):}}\\

\begin{itemize}
\item[3.1] Among the orbitals in $\Top$ with the
smallest value of $\orbd$, find the leftmost orbital
(via Process~\ref{proc:bpcompare}), which we denote $(a,b)$
throughout this step. \\
Let
$Y=\left\{g_1,g_2,\ldots, g_q\right\}$ be the set
of signatures associated to the signed 
orbitals of $SO$ whose orbital is $(a,b)$.
\end{itemize}

(The current occurrence
of Step~3.1 is the start of 
the next step in our breadth-first-search.)

{\flushleft{\it  \underline {Step $3$a} (Building a 
  local controller $c$ over $(a,b)$):}}\\

\begin{itemize}
\item[3.2]  For $1 \leq i \leq q$:
Compute the slopes $m_{g_ia}\defeq ag_+'$ and 
${m_{g_ib}:=bg_-'}$ (Process~\ref{proc:endptslope}).
Let 
$
M_{aY} \defeq\left\{m_{ga}\mid g\in Y\right\}
$
 and 
$
M_{bY}\defeq  \left\{m_{gb}\mid g\in Y\right\}.
$
 
\item[3.3] Determine whether each of 
the groups $\Pi_{M_{aY}}=\langle M_{aY}\rangle$ and 
$\Pi_{M_{bY}}=\langle M_{bY} \rangle$ is discrete
(Process~\ref{proc:euclid}).  
If either of these groups is not discrete, 
terminate the algorithm and output 
``The group $G$ is not soluble.''

\item[3.4] Using Process~\ref{proc:euclid} again, compute
integers $p_1, p_2, \ldots,p_q$ such 
that $m_{g_1a}^{p_1}m_{g_2a}^{p_2}\cdots m_{g_qa}^{p_q}$
is the least real number $\Pi_{M_{aY},s}$ greater than 1
in the group $\Pi_{M_{aY}}$.  
Compute the element 
$
c:=g_1^{p_1}g_2^{p_2}\ldots g_q^{p_q}.
$
(Process~\ref{proc:gpops}).
\end{itemize}

Note that by construction, $c$ has an orbital
with $a$ at its left endpoint.  

\begin{itemize}
\item[3.5]  
Determine whether the orbital of $c$
realising $a$ and the orbital $(a,b)$ of $\Top$ form
a \bado \ (Processes~\ref{proc:tuple} and~\ref{proc:bpcompare}).
If so, then terminate the
algorithm and output ``The group $G$ is not soluble.''

\item[3.6]
Compute the slope $m_{cb}:=bc_-'$ of $c$ at the right endpoint $b$ 
of the orbital $(a,b)$ using Process~\ref{proc:endptslope}.
Also compute the least real number $\Pi_{M_{bY},s}$ greater than 1
in the group $\Pi_{M_{bY}}$ (Processes~\ref{proc:euclid} and~\ref{proc:slopecompare}).
If $m_{cb} \neq \Pi_{M_{bY},s}$, then
terminate the algorithm and output 
``The group $G$ is not soluble.''


\item[3.7]  
For $1\leq i\leq q$:
By iterating over successively larger positive and
negative integers $k$, computing $c^{k}$ (Process~\ref{proc:gpops}) and 
its slope $a(c^k)_+'$ to the right of $a$ (Process~\ref{proc:endptslope}),
and comparing this slope to $m_{g_ia}$ (Process~\ref{proc:slopecompare}),
find the unique integer
$l_{ia}$ such that the slope
$m_{g_ia}$ is equal to the slope $a(c^{l_{ia}})_+'$.
If $b(c^{l_{ia}})'_{-}$ is not equal to $m_{g_ib}$,
then terminate the algorithm and output 
``The group $G$ is not soluble.''
\end{itemize}

(Note that the justification for the outputs
of Steps~3.3,~3.6, and~3.7 is given below after the
completion of the algorithm.) 

{\flushleft{\it  \underline {Step $3$b} (Altering the orbital data sets):}}\\

\begin{itemize}
\item[3.8]  Add the signed orbital $((a,b),c)$ to $SO$.
\end{itemize}

Since all of the factors in the formula 
defining $c$ in Step 3.4 realise the orbital $(a,b)$,
Lemma~\ref{lem:nonsplitrep}(1) says that
$c$ is a one-bump factor of an element of $G$.
Hence Step 3.8 preserves properties $SO$.1 and $SO$.2
of the set $SO$.  


\begin{itemize}
\item[3.9]  For $1 \le i \le q$:  Calculate the element
$h_i:= g_i c^{-l_{ia}}$ of $S(G)$ (via Process~\ref{proc:gpops}).  
Use Process~\ref{proc:tuple}
to build the orbital tuple 
$$
X_{h_i}=[B_{i1}, B_{i2}, \ldots, B_{ik_i}]
$$ 
for $h_i$.  For each orbital 
$B_{ij}$ use Process~\ref{proc:split} to
produce the factor signed orbitals
$(B_{ij},\widetilde h_{ij})$, where
$\widetilde h_{ij}$ is the one-bump function 
which agrees with $h_i$ over $B_{ij}$.
Add the signed orbitals $(B_{ij},\widetilde h_{ij})$
to $SO$, and add the orbitals $B_{ij}$ to $\unseen$.
\end{itemize}


Note that we have $\overline{B_{ij}}\subset (a,b)$ for all $i,j$.
Each of the factors in the product $g_i c^{-l_{ia}}$ 
defining $h_i$ in Step 3.8 realise the orbital $(a,b)$,
and so Lemma~\ref{lem:nonsplitrep}(1) says that
each of the one-bump factors $\widetilde h_{ij}$ of $h_i$
is also a one-bump factor of an element of $G$.
Hence Step 3.9 also preserves properties $SO$.1 and $SO$.2
of the set $SO$.

(Note that in both Steps 3.8 and 3.9,
we may not be adding new signed orbitals 
to $SO$ each time; it may be the case, for
example, that the signed orbital already 
lies in $SO$ due to other generators of $G$.)

\begin{itemize}
\item[3.10]  For $1 \le i \le q$:  Determine whether
$g_i = c$ (using a combination of 
Processes~\ref{proc:bp},~\ref{proc:bpcompare},~\ref{proc:endptslope},
and~\ref{proc:slopecompare}).
If not, then remove
the signed orbital $((a,b),g_i)$ from $SO$.
\end{itemize}

Since $g_i=h_ic^{l_{ia}}$ and $h_i,c \in \siggp(SO)$
after Step~3.10 is applied, the group $\siggp(SO)$
is not altered in this step.  Hence
$SO$.1 and $SO$.2 are again preserved.

Note that the signed orbital $((a,b),c)$ is now
the only element of $SO$ with support $(a,b)$.

{\flushleft{\it  \underline {Step $3$c} (Checking for 
   {\bado}s beneath $c$):}}\\

\begin{itemize}
\item[3.11]  Determine, using Procedure~\ref{proc:bpcompare},
whether or not 
$SO$ contains a \bado.
If so, then terminate the algorithm and output 
``The group $G$ is not soluble.''

\item[3.12]  Compute the set $Proj_{(a,b)}$ 
of all of the signed orbitals in $SO$ with support 
in $(a,b)$ which do not realise the orbital $(a,b)$
(Process~\ref{proc:bpcompare}).
If $Proj_{(a,b)} = \emptyset$, then go to Step 1.
\end{itemize}

Note that since $SO$ has no {\bado}s (after Step 3.11),
each orbital $((r,s),g)$ in $Proj_{(a,b)}$ satisfies
$a<r<s<b$.  Since $c$ has $(a,b)$ as its only orbital
and the slope $ac_+'$ is greater than 1, we also have
$r<r \cdot c$.

\begin{itemize}
\item[3.13]  
Determine whether the orbitals $((r,s),g)$ and 
$((r\cdot c, s\cdot c), g^c)$
form a \bado; to do this, it suffices to check, using 
Processes~\ref{proc:bpeval} and~\ref{proc:bpcompare},
whether $r\cdot c< s$.
If so, terminate the algorithm and output
``The group $G$ is not soluble.''

\end{itemize}

Note that if the algorithm continues after
Step~3.13, then since $x \cdot c > x$ for
all $x \in (a,b)$, we have that
each $((r,s),g) \in Proj_{(a,b)}$
satisfies $(r,s) \subseteq (r,r \cdot c)$.

\begin{itemize}
\item[3.14]  
For each pair of elements $\rho=((r,s),g),\sigma=((u,v),h)$ of $Proj_{(a,b)}$:
By iteratively computing $u\cdot c^i$ for positive and negative integers $i$
(Processes~\ref{proc:gpops} and~\ref{proc:bpeval}) and comparing with $r$ (Process~\ref{proc:bpcompare}),
compute the unique integer $k_{\rho\sigma}$ such that
$r \le u \cdot c^{k_{\rho\sigma}}
   < r\cdot c.$
Construct the signed orbitals 
$\tau_{\rho\sigma}' :=
((u ,v) \cdot c^{k_{\rho\sigma}-1}),h^{c^{k_{\rho\sigma}-1}})$
and 
$\tau_{\rho\sigma} :=
((u ,v) \cdot c^{k_{\rho\sigma}}),h^{c^{k_{\rho\sigma}}})$
(and store them for use in later steps).
Use Process~\ref{proc:bpcompare} to determine 
whether $\tau_{\rho\sigma}'$ or
$\tau_{\rho\sigma}$ yield a \bado\ with $\rho$. 
If so, terminate the algorithm and output
``The group $G$ is not soluble.''
\end{itemize}

Note that continuation of the algorithm after Step~3.14
implies that one of the following must
hold:

$\begin{array}{ll}
\mathrm{(i)}\ (u,v)\cdot c^{k_{\rho\sigma}} = (r,s) &
\mathrm{(ii)}\ \overline{(u,v)}\cdot c^{k_{\rho\sigma}}  \subset (r,s)\\ 
\mathrm{(iii)}\ (u,v)\cdot c^{k_{\rho\sigma}} \subset (s,r \cdot c)  &
\mathrm{(iv)} (u,v)\cdot c^{k_{\rho\sigma}} \supset \overline{(r \cdot c,s \cdot c)}.
\end{array}$

\noindent In case (iv) we have 
  $\overline{(r,s)} \cdot c^{-k_{\rho\sigma}+1}
  = \overline{(r,s)} \cdot c^{k_{\sigma\rho}} \subset (u,v)$.

\begin{itemize}
\item[3.15]  
Determine a partial order $\prec$ on $Proj_{(a,b)}$ as follows.
For each ordered pair of elements $\rho=((r,s),g),\sigma=((u,v),h)$ of $Proj_{(a,b)}$:
Determine whether $v \cdot c^{k_{\rho\sigma}} < s$ 
(and hence whether the unsigned
orbital associated to $\tau_{\rho\sigma}$ satisfies
$\overline{(u,v)}\cdot c^{k_{\rho\sigma}}  \subset (r,s)$), using
Process~\ref{proc:bpcompare} and the stored $\tau_{\rho\sigma}$.
If so, we add 
$\sigma \prec \rho$ to the relation.\\
After this has been completed for all ordered
pairs, determine a leftmost element $\rho_0 = ((r_0,s_0),g)$ of
$Proj_{(a,b)}$ that is maximal with respect to
the relation $\prec$. 
(The choice of $\rho_0$ might not be unique, because two distinct
maximal signed orbitals might share the same support.)
\end{itemize}

To see that the relation $\prec$ is antisymmetric,
suppose that $\rho$ and $\sigma$ are as in Step~3.15
with $\rho\prec \sigma\prec \rho$.  Then 
$\overline{(r,s)}\cdot c^{k_{\sigma\rho}+k_{\rho\sigma}}\subset 
\overline{(u,v)}\cdot c^{k_{\rho\sigma}}\subset (r,s)$,
which is impossible since $c$ moves all points to the right 
on $(a,b)$ and so cannot conjugate an orbital in $(a,b)$ inside itself.

From Step~3.13, we have that the support $(r_0,s_0)$ 
of $\rho_0$ satisfies $(r_0,s_0) \subseteq (r_0,r_0 \cdot c)$.
From the note after Step~3.14, maximality of $\rho_0$
implies that every
signed orbital $\sigma=((u,v),h)$ of $Proj_{(a,b)}$
satisfies either 

$
\begin{array}{ll}
\mathrm{(i)}\ (u,v)\cdot c^{k_{\rho_0\sigma}} = (r_0,s_0), &
\mathrm{(ii)}\ \overline{(u,v)}\cdot c^{k_{\rho_0\sigma}}  \subset (r_0,s_0), \text{ or}\\
\mathrm{(iii)}\ (u,v)\cdot c^{k_{\rho\sigma}} \subset (s_0,r_0 \cdot c).&
\end{array}
$

\noindent Therefore 
$(u,v)\cdot c^{k_{\rho\sigma}} \subset (r_0,r_0 \cdot c)$.
That is, for every signed orbital $\sigma \in Proj_{(a,b)}$, 
the support of the 
signed orbital $\tau_{\rho_0\sigma}$
 is contained in the interval $(r_0,r_0 \cdot c)$.

\begin{itemize}
\item[3.16]   
For each element $\sigma=((u,v),h)$ of $Proj_{(a,b)}$:
Add the signed orbital $\tau_{\rho_0\sigma}$ to $SO$,
and add the associated orbital $(u,v)\cdot c^{k_{\rho_0\sigma}}$
to $\unseen$.
\end{itemize}

Now Lemma~\ref{lem:nonsplitrep}(1)
says that powers of $c$ are one-bump factors of
elements of $G$. In Step~3.16, since
$\overline{(u,v)} \subseteq (a,b)$,
Lemma~\ref{lem:nonsplitrep}(2) says that
the signature $h^{c^{k_{\rho_0\sigma}}}$ 
of $\tau_{\rho_0\sigma}$ is also a one-bump
factor of an element of $G$.  Therefore
Step~3.16 preserves properties $SO$.1 and $SO$.2
of the set $SO$.

\begin{itemize}
\item[3.17]  Determine, using Procedure~\ref{proc:bpcompare},
whether or not 
$SO$ contains a \bado.
If so, then terminate the algorithm and output 
``The group $G$ is not soluble.''
\end{itemize}
 
(Step 3.17 is not strictly necessary, since Steps~3.14 and~3.15
guarantee that no new \bado\ is added to $SO$ in 3.16;
we include Step 3.17 to highlight the fact that
$SO$ does not contain a \bado\ in the following steps.)

\begin{itemize}
\item[3.18]
For each unsigned orbital 
$(u,v)\in \unseen$ such that 
there is an element $h\in \ploi$ with $((u,v),h)\in Proj_{(a,b)}$ and 
${(u,v)\cap (r_0,r_0\cdot c)=\emptyset}$: 
Remove all signed orbitals from $SO$ whose
associated unsigned orbital is $(u,v)$,
and remove the unsigned orbital $(u,v)$ 
from $\unseen$.
\end{itemize}

For any orbital $\sigma=((u,v),h)$ removed from
$SO$ in Step~3.18,
the related element $\tau_{\rho_0\sigma}$ 
with signature $h^{c^{k_{\rho_0\sigma}}}$
(added to $SO$ in Step~3.16)
remains in $SO$.  Since $((a,b),c) \in SO$
as well (from Step~3.8), we have
$h^{c^{k_{\rho_0\sigma}}},c \in \siggp(SO)$,
and so $h \in \siggp(SO)$ after Step~3.18
is complete.  That is, Step 3.18 does
not alter the group $\siggp(SO)$, and
so step~3.18 preserves properties $SO$.1 
and $SO$.2 of the set $SO$.

\begin{itemize}
\item[3.19]
Remove the orbital $(a,b)$ from $\unseen$ 
and add it to $\seen$. 
\end{itemize}

Since $\unseen \cup \seen$
had no {\bado}s in Step~3.17,
and no orbitals were added to
this set in the intermediate Step~3.18,
then Step~3.19 preserves property $\seen$.1.
The fact that Step~3.19 preserves
property $\seen$.2 of the set $\seen$
follows from the
fact that Step~3.10 has been performed
for the orbital $(a,b)$ in the current
instance of Step~3, and
property $\seen$.3 follows from Steps~3.13 through~3.18.

\begin{itemize}
\item[3.20]
Proceed again to Step~$1$.
\end{itemize}

\centerline{{\it \underline{End of algorithm}}}


\bigskip

It remains to show that this algorithm
will terminate on every possible input,
and that when it terminates, it outputs
the correct answer.  We begin with the latter.

In Step~0 of this algorithm, a set $SO$
satisfying properties $SO$.1 and $SO$.2 is
computed, and in all subsequent steps
in which the set $SO$ is changed,
namely Steps~3.8,~3.9,~3.16, and~3.18,
these two properties have been
shown to be preserved.  
Therefore the signature group
$\siggp(SO)$ associated to $SO$
satisfies $G \leq \siggp(SO) \leq S(G)$.
Now Corollary~\ref{cor:subsplitder}(1)
says that the derived length of
$G$ equals the derived length of $\siggp(SO)$
throughout the algorithm.

In all steps in which a \bado\ is found 
among the elements of $SO$, 
namely Steps~1.2,~3.11,~3.13,~3.14,
and~3.17, we have that the signature group
$\siggp(SO)$ admits a \bado.
 Corollary~\ref{cor:subsplitder}(2) then
shows that $G$ is not soluble, verifying
the output of these five steps.

In Step~2.2, if $\maxDepth$ is found to
be greater than the number $n$ of 
breakpoints among the finite set of
homeomorphisms in the input to the
algorithm, then the algorithm has
found an orbital $A \in \unseen$ that
is contained in at least $n$ orbitals
in $\seen$.  Now $\unseen \cup \seen$
is the set of unsigned orbitals associated
to the set $SO$ of signed orbitals, and
 at this step we know
that the set $SO$
contains no {\bado}s (from
Step~1.2).  Thus the
$n$ unsigned orbitals in $\seen$ together with $A$
arise from $n+1$ signed orbitals
in $SO$ that form a tower of height $n+1$,
with the orbital associated to $A$ at the ``bottom'',
and so the orbital depth of $A$
with respect to the signature group
$\siggp(SO)$ must be at least $\orbd(A)>n$.
Then Theorem~\ref{thm:towerSolveClass}
implies that the derived length
of $\siggp(SO)$ is at least $n+1$.
Since $G$ and $\siggp(SO)$ have the
same derived length, then
$G$ must have derived length at least $n+1$.
However, Theorem~\ref{main}
shows that if $G$ is soluble,
then its derived length must be at
most $n$.  Hence the ``not soluble'' output
of Step~2.2 is valid.

To show that Step~3.3 is valid,
we consider the subgroup 
$H:=\langle g_1,...,g_q \rangle$ 
of $\siggp(SO)$, with the single orbital $(a,b)$.
Note that if one
of the groups $\Pi_{M_{aY}}$ and 
$\Pi_{M_{bY}}$ of Step~3.3 is not 
a discrete group, then that group is
neither
the trivial group nor isomorphic to $\Zbb$.
Lemmas~3.10 and~3.11 of~\cite{bpasc}
show that in this case the group $H$ is not
balanced.
Corollary~\ref{cor:chainlessIsBalanced} then
shows that $H$ is not soluble.  Thus
the group $\siggp(SO)$ contains a
nonsoluble subgroup, and so also is
nonsoluble.  Therefore $G$ is not
soluble, as required.

Suppose next that the conditions of
Step~3.6 hold, namely that
the slope $m_{cb}=bc_-'$ of $c$ in a
neighborhood to the left of $b$
satisfies $m_{cb} \neq \Pi_{M_{bY},s}$,
where $\Pi_{M_{bY},s}$ is the least 
number greater than 1
in the discrete group $\Pi_{M_{bY}}$.  
Since the algorithm
did not terminate at Step~3.5,
the element $c$ of $\siggp(SO)$
cannot have a fixed point in 
the interval $(a,b)$, and since
its slope $ac_+'$
on the right at $a$ is greater than 1,
we must have $bc_-' > \Pi_{M_{bY},s}$.
Let $d$ be the element
of $\siggp(SO)$ defined by
$d:=g_1^{\tilde p_1} \cdots g_q^{\tilde p_q}$.
Then the support of $d$ lies in $(a,b)$,
and the slope $bd_-'$ of $d$ from
the left in a neighborhood of the
endpoint $b$ is the number $\Pi_{M_{bY},s}$.
If $d$ has a fixed point in the
interval $(a,b)$, then the group
$\siggp(SO)$ has a \bado.
On the other hand, if $d$ does not have a
fixed point in $(a,b)$, then the 
slope $ad_+'$ of $d$ from the right at $a$
must be greater than 1, and so is 
greater than or equal to the slope
$\Pi_{M_{aY},s}=ac_+'$ of $c$ at $a$.
Now the element $cd^{-1}$ of
$\siggp(SO)$ has an orbital of
the form $(a',b)$ for some $a<a'<b$,
and so $\siggp(So)$ again admits a
\bado.  
Theorem~\ref{main} says
that $\siggp(SO)$ is not soluble in both cases,
and so $G$ also is not soluble.
This verifies the output of Step~3.6.

Next suppose that the condition of
Step~3.7 holds; that is, suppose that
there is an index $i$ such that
$l_{ia} \neq l_{ib}$.
Let $d:=g_ic^{-l_{ia}}$.  Then $d \in \siggp(SO)$,
the support of $d$ is a subset of $(a,b)$, 
and $d$ fixes an open neighborhood of $a$.
However, the slope $bd_-'$ of $d$ at
$b$ from the left is not 1.  Therefore
$\siggp(SO)$ again admits a \bado, and 
so $G$ is not soluble, so Step~3.7 is
also valid.

The last output step left to check is
the only step that outputs that the
group $G$ is soluble, namely Step~1.1.
Suppose that the set $\unseen$ is
found to be empty in an occurrence of
Step~1.1.  If $\counter = 0$, and
so the algorithm terminates at the first 
occurrence of Step~1.1, then
$G$ is the trivial group,
and the algorithm correctly 
outputs the value $0$ for the derived length.
On the other hand, suppose that $\counter > 0$, and so this
the algorithm terminates at a later occurrence
of Step~1.1.
From Steps~0.1 and~3.19,
we know that the unsigned orbitals in the
set $\seen$ satisfy properties
$\seen$.1, $\seen$.2, and $\seen$.3.
Thus the set $Z :=\{h \mid (A,h) \in SO\}$
is a finite generating set
of $\siggp(SO)$ that satisfies properties
Z0-Z3 of 
the definition of a set of \nice.
Lemma~\ref{lem:towertree} says that
the group $\siggp(SO)$ is soluble, 
with derived length equal to the height
of the largest tower that can be formed
from orbitals in the set $SO$.
The algorithm adds unsigned orbitals in
order of containment (larger intervals before smaller),
and so the value of $\maxDepth$ from the
last instance of Step~2.2 will be the derived
length of $\siggp(SO)$.
Again using the fact that
$G$ and $\siggp(SO)$ have the
same derived length, this shows
that the output of Step~1.1 is valid.


Finally we turn to the 
proof that the algorithm will terminate
on all possible inputs.
In Step~0.2 of the algorithm, the
set $SO$ is built from the one-bump
factors of the finite set
$\{f_1,...,f_m\}$ of input functions,
and the finite set $\unseen$ 
of associated unsigned orbitals is
created.  The only step in which
the set $\unseen$ is altered
is Step~3; in particular, 
Steps~3.9,~3.16, and~3.18.
Each time Step~3 is performed,
an orbital of smallest value of $\orbd$
is removed from $\unseen$, as well
as possibly some others of greater
orbital depth, and 
a finite (possibly zero) number
of orbitals of strictly
larger depth are added to $\unseen$.
After a finite number of iterations
of Step~3, then, the least value
of $\orbd$ of an element of $\unseen$
must increase or else $\unseen$ must
become empty.  In the case
that the group $G=\langle f_1,...,f_m \rangle$ 
is soluble this implies that
$\unseen$ must eventually
be empty after a finite number
of occurrences of Step~3, causing
the algorithm to terminate at Step~1.1.
In the case that $G$ is not soluble,
this means that either the algorithm
must halt in one of 
Steps~1.2,~3.3,~3.6,~3.7,~3.11,~3.13,~3.14, or~3.17,
or else after a finite number of steps
the smallest value of $\orbd$
among the elements of $\Top \subseteq \unseen$
is greater than the number $n$ of
breakpoints of the $f_i$ input functions,
causing the algorithm to terminate at
Step~2.2.
\end{proof}

\begin{remark} {\em
In the proof of Theorem~\ref{thm:algorithm},
every element of $S(G)$ involved in the
computations throughout the
procedure can be shown to be the one-bump 
factor of an element of $G$, using Lemma~\ref{lem:nonsplitrep}.
If the algorithm stores the element of $G$
with each of these one-bump factors, then
in each application of Processes~1-9, 
it is possible for the algorithm
instead to perform the procedure with
the corresponding elements of $G$, 
and then apply Process~\ref{proc:split},
in order to accomplish the process for
the element of $S(G)$.  At some potential cost in
efficiency, then, Theorem~\ref{thm:algorithm}
also holds for groups $C$ admitting
Processes~1-9 in which the
group $S(C)$ is replaced by $C$ in each of the
process statements.
}\end{remark}

We note that it may be possible 
to make this algorithm  more efficient;
in particular,
some repeated steps may be streamlined.
It is of interest  to consider whether
a different strategy for choosing
the element of $\Top$ to consider in
the next occurrence of Step~3.1, for
example with a depth-first-search instead,
would improve efficiency.
We also note that the algorithm can
be made parallel in various ways,
for example by processing all elements
of least $\orbd$ value in $\Top$ 
simultaneously, while the sets 
$SO$, $\unseen$, and $\seen$  and the value $\maxDepth$
are treated as global objects in shared memory.

Finally, we turn to the solution of the
membership decision problem for finitely
generated soluble subgroups of computable
subgroups of $\ploi$ that are generated
by a finite set of \nice.

\begin{corollary}\label{thm:membership}
Let $C$ be a 
computable subgroup of $\ploi$.
Let $H$ be a subgroup of $C$
generated by a finite set of \nice.
Then the membership decision problem 
is solvable for $H$; that is,
there is an algorithm which, upon input of
an element $w$ of $C$, can determine
whether $w \in H$.
\end{corollary}

\begin{proof}
Let $Z$ be the finite set of \nice\ generating $H$.
For each $h \in Z$, we replace $h$ by
$h^{-1}$, if necessary, so that we
may assume that the slope $m_{ha}=ah'_+$
of $h$ at the left endpoint $a=\inf\supp{h}$
of its support satisfies $m_{ha}>1$.

We first show that $H$ is equal to the split group $S(H)$.
Following the notation of the proof of
Lemma~\ref{lem:towertree}, let $n$ be
the largest height of a tower in the set
$S_Z=\{(A_h,h) \mid h \in Z\}$ 
(where $A_h=\supp{h}$) of signed
orbitals associated to the elements of $Z$.
If $n=0$, then $Z$ is empty and $S(H)=H=1$
is the trivial group.  
If $n=1$ then the elements of $Z$ have
disjoint support, $H$ is the free abelian
group generated by the elements of $Z$,
and again $S(H)=H$.
Now suppose that $n>1$ and the result holds for
finite sets of \nice\ with maximum associated
tower height at most $n-1$.
Suppose that $g'$ is
any one-bump factor of an element $g \in H$. 
Recall from the proof of
Lemma~\ref{lem:towertree} that
$H = \oplus_{h \in Y} (\langle P_h \rangle \wr \Zbb)$
where $Y=\{h' \in Z \mid \orbd(A_{h'})=1\}$
is the set of elements of minimal orbital
depth in $Z$, and such that for each
$h \in Y$ the set
$P_{h}:=\{h' \in H \mid A_{h'} \subsetneq A_h\}$
is the set of elements of $Z$ whose support is
properly contained in the support $A_h$ of $h$.
Since the support of the group $H$ is
$\cup_{h \in Y} A_h$,
we have $\supp{g'} \subseteq A_h$ for some $h \in Y$.
Now the element $g \in \langle Z \rangle$ 
is a product of
an element $\tilde g$ of 
$\langle h,P_h\rangle=\langle P_h \rangle \wr \langle h \rangle$
with an element of 
$\langle Z \setminus (\{h\} \cup P_h) \rangle$ whose
support does not intersect $A_h$.
Moreover, $\tilde g$ is another element of $H$
that has $g'$ as a one-bump factor.
We can write $\tilde g = \hat g h^k$ for some 
$\hat g \in \oplus_{j \in \Zbb} \langle P_h \rangle^{h^j}$ 
and $k \in \Zbb$.
Moreover, $\hat g$ can be written
as a product of elements of
a finite
subset $Q$ of $\cup_{j \in \Zbb} (P_h)^{h^j}$ 
that is a set of \nice\ with maximum associated
tower height at most $n-1$.  
If $k=0$ then 
$g'$ is a one-bump factor of
$\tilde g=\hat g \in \langle Q \rangle$, and so
$g'$ is an element of the split group
$S(\langle Q \rangle)$.  
By the inductive assumption above, 
$S(\langle Q \rangle)=\langle Q \rangle$;
hence in the $k=0$ case,
$g' \in \langle Q \rangle < H$.
On the other hand, if 
$k \neq 0$, then since the
supports of the elements
in $\cup_{j \in \Zbb} (P_h)^{h^j}$ do not
share an endpoint of $A_h$, the support
of $\tilde g$ includes intervals 
with endpoints that are the endpoints of $A_h$.
Since $H$ is soluble (Lemma~\ref{lem:towertree}),
Theorems~\ref{main} and~\ref{thm:splitderlength} 
show that $S(H)$
does not admit a \bado, and so we have
$\supp{\tilde g}=A_h$.   In this case
$\tilde g$ is already a one-bump function, and so
$g'=\tilde g$.  Thus again we have
$g' \in H$.  Hence $S(H)=H$, as claimed.

Next we note that upon input of the
set $Z$ to the  algorithm 
of Theorem~\ref{thm:algorithm},
no orbitals are added or removed from the
set $SO$ after Step 0.2, and
the algorithm will terminate at an instance
of Step 1.1, with $SO=\seen=S_Z$,
and output the derived length $n$ of $H$.

Finally we are ready to give the
MDP algorithm.
Input the set $Z \cup \{w\}$ to 
the SSRP algorithm of Theorem~\ref{thm:algorithm}.  
At step 0.2, the algorithm will
place the signed orbitals of the one-bump factors of $w$
into the set $SO$; since $S(H)=H$,
these factors lie in $H$ iff $w$ lies in $H$.
Proceeding through the algorithm,
if at any time the SSRP algorithm outputs
``The group $G$ is not soluble'' or
``The group $G$ is soluble with derived
length $m$'' where $m$ is greater than
the derived length $n$ of $H$,
then the present (MDP) algorithm
outputs ``The element $w$ is not in $H$''.
For the rest of this proof we assume 
that the eventual output of 
the SSRP algorithm (with input
$Z \cup \{w\}$) is ``The group $G$ is soluble
with derived length $n$''.

The MDP algorithm uses a slight restriction
on Step~3 of the SSRP procedure, to ensure
that no signed orbital associated
to an element of $Z$ is removed from
the set $SO$.
Each time that the SSRP 
algorithm reaches Step~3.1,
since $Z \subseteq SO$,
the breadth-first-search structure of the
SSRP algorithm, processing intervals of least
orbital depth first, guarantees that
for all $h' \in Z$ satisfying $A_{h'} \supsetneq (a,b)$,
Step~3 has already been performed for the interval $A_{h'}$.
Also condition Z2 for the set $Z$
implies that the set $Y$ of elements
of $SO$ with support $(a,b)$ contains at most
one element of $Z$.  
Suppose first that $Y$ does not contain any element
of $Z$;
that is, all elements $w'$ of $Y$
are derived from $w$ via earlier Steps~0.2,~3.9, and~3.16
of the SSRP.
Then the group $\langle Z \cup \{w\} \rangle$
contains a subgroup $\langle Y \rangle$ 
not in $\langle Z \rangle$, and so
the MDP algorithm halts and outputs 
``The element $w$ is not in $H$''.
Next suppose that $Y=\{h\}$ is 
a singleton set whose element $h$ is in
$Z$.  Then the only substep
of Step~3 that has an effect
is Step~3.19, moving the orbital
$A_h=(a,b)$ from $\unseen$ to $\seen$; then
the SSRP algorithm returns to Step~1. 
Finally suppose that $Y$ contains an element $h$
of $Z$ and $|Y|>1$.
If the slope $m_{ha}$ of
the element $h$ at $a$ does not equal the slope 
$\Pi_{M_{aY}}$ in the subsequent
occurrence of Step~3.4,
then the group of slopes at $a$ of signed orbitals
with support $(a,b)$
for $\langle Z \cup \{w\} \rangle$
does not equal the same slope group
for $\langle Z \rangle$, and so
we stop and
output ``The element $w$ is not in $H$''.
Otherwise, we can take $c=h$ in 
this round of Step~3.4.
Continuing with Step~3,
in Steps~3.9 and~3.10
the orbital associated to each $w' \in Y \setminus \{h\}$
in $SO$ is replaced by signed orbitals of one-bump 
factors of a product of $w'$ with
a power of $h$, and in Steps~3.16 and~3.18
these orbitals may be replaced again
by orbitals associated to conjugation
of the signatures by a power of $h$.
Again using 
$S(H)=H$, we have 
$w' \in H$ iff these (conjugates of)
factors lie in $H$.
Again in Step~3.19 the orbital
$A_h=(a,b)$ is moved from $\unseen$ to $\seen$,
and then the SSRP algorithm returns to Step~1.

Continue through the SSRP procedure
and repeat the above process for all
instances of Step~3.
When the SSRP algorithm terminates,
the MDP
algorithm outputs ``The element $w$ is not in $H$''
unless the
SSRP algorithm terminates at an
instance of Step~1.1
with $SO=\seen=S_Z$, in 
which case the output is
``The element $w$ is in $H$''.
\end{proof}

\section*{Acknowledgment}

The third author was partially supported by grants from the Simons
Foundation (\#245625) and the National Science Foundation (DMS-1313559).


\bibliographystyle{amsplain}
\bibliography{go}

\end{document}